\documentclass[a4paper,12pt,oneside,reqno]{amsart}

\usepackage{hyperref}
\usepackage[headinclude,DIV13]{typearea}
\areaset{15.1cm}{25.0cm}
\usepackage{txfonts,amssymb,amsmath,amsthm,bbm}
\usepackage{footmisc}
\usepackage[latin1] {inputenc}
\usepackage{xcolor}
\usepackage{subfigure}
\usepackage{comment}
  
\newtheorem{theorem}{Theorem}[section]
\newtheorem{lemma}[theorem]{Lemma}
\newtheorem{proposition}[theorem]{Proposition}
\newtheorem{corollary}[theorem]{Corollary}
\newtheorem{remark}[theorem]{Remark}

\definecolor{bbm}{RGB}{51,153,0}
\definecolor{above}{RGB}{128,0,128}
\definecolor{below}{RGB}{102,0,204}
\definecolor{cascade}{RGB}{204,0,0}
\definecolor{iid}{RGB}{153,51,0}

\def\paragraph#1{\noindent \textbf{#1}}

\numberwithin{equation}{section}

\def\Var{\mathop{\rm Var}\nolimits}
\def\Cov{\mathop{\rm Cov}\nolimits}

\def\<{\langle}
\def\>{\rangle}

\def\r{\rho}

\let\cal=\mathcal

\def\II{{\cal I}}

 \def \r {{\rho}}

 \def \ba {\begin{array}}
 \def \ea {\end{array}}

 \newcommand{\be}{\begin{equation}}
 \newcommand{\ee}{\end{equation}}

\newcommand{\bea}{\begin{eqnarray}}
 \newcommand{\eea}{\end{eqnarray}}
\def\TH(#1){\label{#1}}\def\thv(#1){\ref{#1}}
\def\Eq(#1){\label{#1}}\def\eqv(#1){(\ref{#1})}

 \def \1{\mathbbm{1}}

\newcommand{\bbL}{\mathbb{L}}
\newcommand{\bbT}{\mathbb{T}}

\newcommand{\bbV}{\mathbb{V}}

\newcommand{\bbN}{\mathbb{N}}

\newcommand{\bbR}{\mathbb{R}}

\newcommand{\ol}{\overline}

\newcommand{\ul}{\underline}

\newcommand{\cI}{\mathcal I}

\newcommand{\cN}{\mathcal N}

\newcommand{\rmd}{{\rm d}}
\newcommand{\rme}{{\rm e}}

\newcommand{\rmB}{{\rm B}}

\newcommand{\rmP}{{\rm P}}

\newcommand{\rmE}{{\rm E}}
\newcommand{\rmVar}{{\rm Var}}
\newcommand{\rmCov}{{\rm Cov}}
\newcommand{\loc}{{\rm loc}}

\usepackage{xr}
\begin{document}


 \title[DGFF Subject to a hard wall / asymptotics]
{Gaussian Free Field on the Tree Subject to a Hard Wall \\II: Asymptotics}
\author[M. Fels]{Maximilian Fels}
\address{M. Fels, 
	Technion - Israel Institute of Technology. Haifa, 3200003,
	Israel.}
\email{felsm@campus.technion.ac.il}
\author[L. Hartung]{Lisa Hartung}
\address{L. Hartung, Universit\"at Mainz. 
	Staudingerweg 9,
	55128 Mainz, Germany.}
\email{lhartung@uni-mainz.de}
\author[O. Louidor]{Oren Louidor}
\address{O. Louidor,
	Technion - Israel Institute of Technology. Haifa, 3200003,
Israel.}
\email{oren.louidor@gmail.com}

\date{\today}

\begin{abstract}
This is the second in a series of two works which study the discrete Gaussian free field on the binary tree when all leaves are conditioned to be positive. In the first work~\cite{Work1} we 
identified the repulsion profile followed by the field in order to fulfill this ``hard-wall constraint'' event. In this work, we use these findings to obtain a comprehensive, sharp asymptotic description of the law of the field under this conditioning. We provide asymptotics for both local statistics, namely the (conditional) law of the field in a neighborhood of a vertex, as well as global statistics, including the (conditional) 	law of the minimum, maximum, empirical population mean and all subcritical exponential martingales. We conclude that the laws of the conditional and unconditional fields are asymptotically mutually singular with respect to each other.
\end{abstract}

\maketitle

\tableofcontents

\section{Introduction and Results}
\subsection{Setup and recap}
\label{s:1.1}
The discrete Gaussian free field (DGFF) on the infinite binary tree $\bbT$ rooted at $0$ (with Dirichlet boundary conditions) is a centered Gaussian process $h = (h(x) :\: x \in \bbT)$ with covariances given by
\begin{equation}
\label{e:01.1}
\rmE\big[h(x) h(y)\big] = \tfrac12 \big(|x| + |y|) - \rmd_{\bbT}(x,y)\big) = |x \wedge y|\,.
\end{equation}
Above $\rmd_{\bbT}$ is the graph distance on the tree, $|x| := \rmd_{\bbT}(x,0)$ is the depth of $x \in \bbT$ and $x \wedge y$ denotes the deepest common ancestor of $x,y \in \bbT$. Alternatively, $h$ can be seen as a branching random walk (BRW) with fixed binary branching and standard Gaussian steps, in which case
we shall use the term generation instead of depth. We shall refer to $h$ as both a DGFF and a BRW interchangeably. We denote by $\bbT_n$, resp. $\bbL_n$ the sub-graph of $\bbT$ which includes all vertices at depth (or generation), at most, resp. equal to, $n \geq 0$. The law of $h$ on $\bbT_n$ will be denoted by $\rmP_n$, with $\rmP_\infty \equiv \rmP$ being the law of the full process on $\bbT$.

This is the second in the line of two works whose aim is to study the law of the field $h$ conditional on the, so-called, \textit{hard-wall event}:
\begin{equation}
\label{e:1}
\Omega_n^+ := \big\{ h(x) \geq 0 :\: x \in \bbL_n \big\} \,.
\end{equation}
The motivation for considering this particular problem, its context and relation to previous works is provided in the first work in the series~\cite{Work1}. The latter work also lays the foundations for the derivations in the present manuscript, in the form of results and tools which are made use of here. In general, while this paper is self-contained, we invite the reader to read~\cite{Work1} first and have it close at hand, while going over this work.

We recall that in~\cite{Work1} we showed \cite[Corollary~\ref{1@c:1.4}]{Work1} that under the conditional law, $\rmP_n^+(\cdot) := \rmP_n(\cdot\,|\, \Omega_n^+)$,
the mean height of the field at $x \in \bbT_n$ is, up to $\Theta(1)$,
\begin{equation}
\label{e:1.10a}
	\mu_n(x) :=  m_{n'} \big(1 - 2^{-|x|}\1_{\{|x| < l_n\}}\big) \,,
\end{equation}
where
\begin{align}\label{Lisa.1}
m_n=c_0 n- \tfrac{3}{2}c_0^{-1}\log n \quad ; \qquad c_0=\sqrt{2\log 2} \,.
\end{align}
and
\begin{equation}
l_n :=  \lfloor \log_2 n\rfloor \ \ , \ \ \ \ n' := n - l_n  \ \ , \ \ \ 
	m_{n'} = m_n - c_0 \lfloor \log_2 n\rfloor + o(1) \,.
\end{equation}
For $|x| \leq l_n$, the height concentrates around this mean with Gaussian lower tails and ``almost-Gaussian'' upper tails~\cite[Theorem~\ref{1@t:1.3}]{Work1}. For $|x| > l_n$, the fluctuations are Gaussian with variance $\Theta(|x|-l_n)$. Moreover, the covariance between $x,y \in \bbT_n$ with $|x|,|y| \geq l_n$ but $|x \wedge y| < l_n$ decays exponentially fast in $l_n - |x \wedge y|$ with $l_n$ as above~\cite[Theorem~\ref{1@t:1.4}]{Work1}. In particular, the heights of the conditional field at vertices in generation $l_n$ are ``almost i.i.d.'', and tight around $m_{n'}$ (see~\cite{Work1} or Subsection~\ref{s:3.2} here for a precise formulation of these statements.) 

Using the Markov property of $h$, which is preserved under $\rmP_n^+$, we thus see that passed depth $l_n$, the conditional field has essentially the same law as the union of $|\bbL_{l_n}| = 2^{l_n} = \Theta(n)$ independent DGFFs on a tree of depth $n-l_n = n'$, each starting at a random height, which is tight around $m_{n'}$, and each conditioned to stay above $0$. Since $-m_{n'}$ is the typical height of the global minimum of each of these DGFFs, this conditioning remains non-singular in the limit.

In this manuscript we go beyond the above rough picture and derive a rather comprehensive asymptotic description of the law for the field on $\bbT_n \setminus \bbT_{l_n-1}$ under $\Omega_n^+$. The discussion above suggests that in order to provide such an asymptotic description, it would be sufficient to derive asymptotics for the law of $h - m_n'$ at $x \in \bbL_{l_n}$ under $\rmP_n(-|\Omega_n^+)$, and also, {\em separately}, asymptotics for the law of $h$ on $\bbT_{n'}$ under $\rmP_{n'}(-|\Omega_{n'}(u))$ for $u=\Theta(1)$, where $\Omega_n(u)$ is the right tail event
\begin{equation}
\label{e:01.14}
	\Omega_n(u) := \Big\{	\min_{\bbL_n} h \geq -m_n + u \Big\}
	\quad ; \qquad u \in \bbR \,.
\end{equation}
The marginal law of the centered conditional field is thus that of a dependent sum of two random variables: one which is a tight and one which may have unbounded fluctuations (for $|x| \gg l_n)$.

To state asymptotics for the law of such random variables, without scaling which would hide the $\Theta(1)$ terms, one needs to measure distances between probability measures using a metric which could ``separate'' between the law of non-tight random variables and that of their translate by a $\Theta(1)$ quantity. In particular, we note that for our purposes using the usual weak topology on probability measures, via, e.g. the Levy-Prohorov metric, is not sufficient. Under this metric, the distance between, e.g., a centered Gaussian with variance $n$ and the same Gaussian but with mean $\mu$ will tend to $0$ as $n \to \infty$ for all $\mu \neq 0$, thus making these laws indistinguishable in the limit. Indeed, an analogous comparison appears, e.g., in Theorem~\ref{p:107.5}.

Our asymptotic statements will therefore be expressed using the strong Wasserstein metric of order $p \geq 1$ between probability measures on $\bbR^{\bbT}$, with the latter space equipped with the (possibility infinite) $\ell^{\infty}$ distance. Denoting this metric by $W^{p,\infty}$ we recall that
\begin{equation}
	W^{p,\infty}\big(\rmP^{(1)}, \rmP^{(2)}\big) = \inf_{\ddot{\rmP}^{(1),(2)}} \,\Big(\ddot{\rmE}^{(1),(2)} \big\|h^{(1)}- h^{(2)}\big\|_\infty^p\Big)^{1/p} \,,
\end{equation}
where the infimum is over all laws $\ddot{\rmP}^{(1),(2)}$ on $\bbR^{\bbT} \times \bbR^{\bbT}$ having marginals $\rmP^{(1)}$, $\rmP^{(2)}$, with $h^{(1)}$, $h^{(2)}$ denoting the canonical projection processes and $\ddot{\rmE}^{(1),(2)}$ the associated expectation. We remark that we allow both the $\ell^{\infty}$ norm as well as the above distance to be $\infty$.

We shall use this distance also for probability measures $\mu$ on $\bbR^A$ with $A \subset \bbT$, by treating such measures also as measures on $\bbR^{\bbT}$ via the natural extension
\begin{equation}
	\tilde\mu \big(\{h :\: h_{\bbT \setminus A} = 0\big\}\big) = 1
	\quad, \qquad
	\tilde\mu \circ \Pi_A^{-1} = \mu \,,
\end{equation}
where $\Pi_A$ denotes the projection mapping from $\bbR^{\bbT}$ onto $\bbR^A$ and where henceforth, we will use both $f_B$ and $f(B)$ to denote the restriction of a function $f$ to the subset $B$ of its domain.

We shall say that a sequence of probability measures $(\mu_n)_{n \geq 1}$ on $\bbR^\bbT$ tends {\em locally} in $W^{p,\infty}$ to a probability measure $\mu$ on $\bbR^{\bbT}$ if 
\begin{equation}
\label{e:1001.9}
	W^{p,\infty} \big(\mu_n \circ \Pi^{-1}_{\bbT_k},\, \mu \circ \Pi^{-1}_{\bbT_k}\big) \underset{n \to \infty}\longrightarrow 0 
	\quad, \forall\, k \geq 0 \,.
\end{equation}
This convergence defines a topology on probability measures on $\bbR^{\bbT}$ which is metrizable, and we denote by $W^{p,\infty}_{\rm loc}$ a metric such that $W^{p,\infty}_{\rm loc}(\mu_n, \mu) \longrightarrow 0$ as $n \to \infty$ if and only if~\eqref{e:1001.9} holds.

We note that for any $p \geq 1$, the topology of $W^{p,\infty}$ is stronger than that of $W^{p,\infty}_\loc$, which is in turn stronger than the topology of weak convergence for probability measures on $\bbR^{\bbT}$ with respect to the underlying product topology. Convergence under $W^{p,\infty}_\loc$ therefore implies weak convergence of all marginals.
We also note that an added benefit of using the $W^{p,\infty}$ topology is that an asymptotic equivalence of measures under $W^{p,\infty}$ induces an asymptotic equivalence of the corresponding moments of order $p$. Lastly, as customary, if $h^{(1)}$, $h^{(2)}$ have laws $\rmP^{(1)}$ and $\rmP^{(2)}$ respectively, we shall write $W^{p,\infty}(h^{(1)}, h^{(2)})$ to mean $W^{p,\infty}(\rmP^{(1)}, \rmP^{(2)})$ with a similar convention for $W^{p, \infty}_\loc$.

In providing an asymptotic description for $h$ under $\rmP_n^+$, we will treat both local and global statistics of the field. By local statistics, we refer to the (joint) law of the field at an $o(n)$- neighborhood (in graph distance) of a given vertex in $\bbT_n$. By global statistics, we refer to law of quantities which involve all vertices of $h$, such as the global maximum or minimum of the field. We begin with the former.

\subsection{Results: local statistics}
As mentioned above, the law of the field conditioned on the right tail event $\Omega_n(u)$ will play an important role in the derivation. To this end, for $u,v \in \bbR$ and $n \in [1,\infty]$, we let 
\begin{equation}
	\rmP_{n,v}(h \in \cdot) \equiv \rmP_{n} \big(h \in \cdot \,\big|\, h(0) = v\big) := \rmP_n \big(h - v \in \cdot \big) \,,
\end{equation}
and
\begin{equation}
\label{e:107.30a}
	\rmP_{n,v}^{\uparrow u}(h \in \cdot) \equiv \rmP_{n} \big(h \in \cdot \,\big|\, h(0) = v,\, \Omega_n(u)\big) := \rmP_n\big(h \in \cdot - v \,\big|\,\Omega_n(u-v)\big)  \,.
\end{equation}
These are the laws of the DGFF starting at $h(0)=v$ unconditionally and conditioned on its minimum to be above $-m_n + u$. Notice that the above definition implies
\begin{equation}
\label{e:101.20}
	\rmP_{n,v}^{\uparrow u}(h \in \cdot) 
= 	\rmP_{n}^{\uparrow u-v}(h-v \in \cdot) \,.
\end{equation}
We shall occasionally omit $u$ and/or $v$ if they are zero and sometimes write $\rmP$ or $\rmP^{\uparrow u}$ in place of $\rmP_\infty$ and $\rmP^{\uparrow}_\infty$.

Our first result shows that $\rmP_{n,v}^{\uparrow}$ admits an infinite volume (or generation) limit. As in~\cite{Work1}, for $u \in \bbR$, $n \geq 1$ we let
\begin{equation}
\label{e:3.3}
p_{n}(u) := \rmP_n\big(\Omega_n(u)\big) \,,
\end{equation}
and
\begin{equation}
\label{e:1.1a}
	p_\infty(u) := \lim_{n \to \infty} p_n(u) \,,
\end{equation}
where the limit exists thanks to the convergence of the centered minimum \cite{AidekonBRW13}. As in~\cite{Work1} if $x \in \bbT$ we write $\bbT(x)$ for the subtree of $\bbT$ rooted at $x$. If $A \subset \bbT$ we shall write $A(x)$ for the projection of $A$ under some (arbitrary) isomorphism between $\bbT$ and $\bbT(x)$, which maps $0 \in \bbT$ to $x \in \bbT(x)$. In particular, the sets $\bbT_k(x)$ and $\bbL_k(x)$ include all descendants of $x$ in $\bbT(x)$ which are at distance at most, resp. exactly, $k \geq 0$ from $x$.
\begin{proposition}
\label{p:1.45}
For all $u,v \in \bbR$, the limit 
\begin{equation}
\label{e:5.69}
\rmP_v^{\uparrow u} \equiv \rmP_{\infty,v}^{\uparrow u} := \lim_{n \to \infty} \rmP^{\uparrow u}_{n,v}
\end{equation}
exits in $W^{p,\infty}_\loc$ for all $p \geq 1$. The limit satisfies for all $r \geq 1$ 
\begin{equation}
\label{e:105.70}
	\frac{\rmd \rmP_v^{\uparrow u}\big(h_{\bbT_r} \in \cdot)}{\rmd \rmP_{r,v}}
	= \frac{\prod_{x \in \bbL_r} p_{\infty} \big(u - c_0 r - h(x))}
		{\rmE_{r,v} \prod_{x \in \bbL_r} p_{\infty} \big(u - c_0 r - h(x))}
		\quad \text{a.s.}\,,
\end{equation}
where $p_\infty$ is as in~\eqref{e:1.1a} and for all $x \in \bbT$, $u,v \in \bbR$\,,
\begin{equation}
\label{e:107.27}
\rmP_v^{\uparrow u}\big(h_{\bbT(x)} \in \cdot \,\big|\, h_{\{x\} \cup (\bbT \setminus \bbT(x))} \big) 
= \rmP^{\uparrow u - c_0 |x|}_{h(x)} \quad \text{a.s.} \
\end{equation}
Finally, for $u,v \in \bbR$ it also holds that
\begin{equation}
\label{e:107.28}
	\rmP_v^{\uparrow u} \big(h \in \cdot) = 	\rmP_0^{\uparrow u - v}\big(h \in \cdot - v\big) \,.
\end{equation}
\end{proposition}
As before, we shall occasionally omit $v$ if it is zero and also write $\rmP^{\uparrow u}(\cdot\,|\, h(0) = v)$ in place of $\rmP_v^{\uparrow u}$. We note that Relation~\eqref{e:107.27} implies that for for any $u \in \bbR$, the collection
$(v \mapsto \rmP^{\uparrow u - c_0 |x|}_{v} :\: x \in \bbT)$ forms a Markovian family of kernels, while Relation~\eqref{e:107.28} links between these families for different $u$-s. Relation~\eqref{e:105.70} recasts $\rmP^{\uparrow u}_v$ as a Doob $h$-transform of the measure $\rmP_{v,\infty}$. It is also an infinite volume Gibbs measure corresponding to the DGFF Hamiltonian, and as such satisfies the DLR conditions with respect to that Hamiltonian (see also Subsection~\ref{1@s:2.2} in ~\cite{Work1}). 

Using $\rmP^{\uparrow u}_v$ we can define the asymptotic law of $h-m_{n'}$ on $\bbT_{n'}(x)$ for $x \in \bbL_{l_n}$ under $\rmP_n^+$. As in~\cite{Work1}, for $n \in \bbN \setminus \{0\}$ we let
\begin{equation}
\label{e:1.7}
[n]_2 := \log_2 n - \lfloor \log_2 n \rfloor \,.
\end{equation}
It is easy to see that $[n]_2$ is always in $[0,1)$ and {\em log-dyadic}, which we henceforth define as a real number which can be written as $\log_2 (q/2^p)$ for some $q \in \bbN \setminus \{0\}$ and $p \in \bbN \cup \{0\}$. 
\begin{theorem}
\label{t:1.7}
For all log-dyadic $\delta \in [0,1)$ with $x_k \in \bbL_k$, the limit
\begin{equation}
\label{e:1.17}
\rmP^{+,\delta} := \lim_{k \to \infty} \rmP^{\uparrow c_0 k} \Big(h_{\bbT(x_k)} \in \cdot \,\Big|\, h(0) = -c_0 2^{k+\delta} \Big)\,,
\end{equation}
exists in $W_\loc^{p, \infty}$ for all $p \geq 1$ and satisfies for all $x \in \bbT$, 
\begin{equation}
\label{e:107.30}
\rmP^{+,\delta} \big(h_{\bbT(x)} \in \cdot \,\big|\, h_{\{x\} \cup (\bbT \setminus \bbT(x))} \big) 
= \rmP^{\uparrow (-c_0 |x|)}_{h(x)} \quad \text{a.s.}
\end{equation}
In particular,
\begin{equation}
\label{e:107.31}
	\rmP^{+,\delta}\big(\cdot) = \int \rmP^{\uparrow}_v\big(\cdot\big)\, \rmP^{+,\delta} \big(h(0) \in \rmd v\big) \equiv \rmP^{\uparrow}_{\rmP^{+,\delta} (h(0) \in \cdot)} \,.
\end{equation}
\end{theorem}
\noindent
The above theorem implies that the law $\rmP^{+,\delta}$ can be recast as the law of the Markov process $\rmP^{\uparrow}_v \equiv \rmP^{\uparrow 0}_v$ from Proposition~\ref{p:1.45} with a random initial height $v \sim \rmP^{+,\delta} \big(h(0) \in \cdot \big)$, with law which is the $W^{p,\infty}$-limit in~\eqref{e:1.17} with $h(x_k)$ in place of $h_{\bbT(x_k)}$.

The next theorem shows that $\rmP^{+,[n]_2}$ is indeed the asymptotic law of $h$ under $\rmP_n^+$. This asymptotic equivalence is stated in a rather strong sense, both with regard to the topology used on the space of probability measures and the subsets of $\bbT$ to which these laws are restricted.
\begin{theorem}
\label{t:7.7}
Fix $p \in [1,\infty)$ and, for all $n \geq 1$, let $A_n \subset \bbT_{n'}$ be connected. Suppose that
\begin{equation}
\label{e:107.40}
	\lim_{r \to \infty} \sup_{n} \sum_{k \geq r} \big|A_n \cap \bbL_k \big| 2^{-k/p} k = 0\,.
\end{equation}
Then with $x_{l_n} \in \bbL_{l_n}$,
\begin{equation}
\label{e:107.41}
W^{p,\infty} \Big(h_n^+\big(A_n(x_{l_n})\big) - m_{n'} 
\,,\,\, h^{+,[n]_2}(A_n) \Big) \underset{n \to \infty}{\longrightarrow}
	0 \,,
\end{equation}
where $h_n^+$ has law $\rmP_n^+$ and $h^{+,[n]_2}$ has law $\rmP^{+,[n]_2}$.
\end{theorem}

\begin{remark}
The above theorem suggests that the asymptotic law of $h - m_{n'}$ under the conditioning on $\Omega_n^+$ depends on $\delta = [n]_2$. This dependency, which appears in the asymptotic statements for all local and global observables below, comes out in the proof, in which we study the field in the first $\lfloor n \rfloor_2 = \log_n - [n]_2$ generations. It is thus an artifact of the proof, which, in turn, stems from the discreteness of the time evolution/field's domain in the model. Nevertheless, we could not rule out (or in) the possibility that this dependency is degenerate, namely that $\rmP^{+,\delta}$ (and all $\delta$-dependent quantities in the statements to follow) does not depend on $\delta$, or trivial, by which we mean that $\rmP^{+,\delta}(\cdot) = \rmP^{+,0}(\cdot - \theta_\delta)$ for some $\theta: [0,1) \to \bbR^{\bbT}$ (so that $\theta_{[n]_2}$ could be added to the centering sequence $m_{n'}$ in~\eqref{e:107.41}, resulting in a constant asymptotic law). This issue, which 
concerns the question of ``survival'' of the discreteness of the model in the limiting law of the conditional field, warrants further investigation, but falls outside the scope of this work.
\end{remark}

By choosing particular sequences of sets $(A_n)_{n \geq 1}$, we obtain the following consequences of Theorem~\ref{t:7.7}.
\begin{corollary}
\label{c:1.9}
Fix $p \in [1,\infty)$. If $(A_n)_{n \geq 1}$ with $A_n \subset \bbT_{n'}$ satisfies
\begin{equation}
\label{e:103.32a}
	\sum_{k \geq 0} \big|A_n \cap \bbL_k\big| 2^{-k/p} k \underset{n \to \infty}\longrightarrow 0\,,
\end{equation}
or
$A_n = A \cap \bbT_{n'}$, for some $A \subset \bbT,$ satisfying 
\begin{equation}
\label{e:103.32}
	\sum_{k \geq 0} \big|A \cap \bbL_k\big| 2^{-k/p} k < \infty \,,
\end{equation}
then~\eqref{e:107.41} holds. 
\end{corollary}

Since~\eqref{e:103.32} clearly holds for any finite set $A \subset \bbT$, we immediately get
\begin{proposition}
\label{p:1.7e}
Fix $p \in [1,\infty)$. With $x_{l_n} \in \bbL_{l_n}$,
\begin{equation}
\label{e:107.58}
		W^{p,\infty}_\loc \Big(h_n^+\big(\bbT_{n'}(x_{l_n})\big) - m_{n'} 
\,,\,\, h^{+,[n]_2} \Big) \underset{n \to \infty}{\longrightarrow}
	0 \,,
\end{equation}
where $h_n^+$ has law $\rmP_n^+$ and $h^{+,[n]_2}$ has law $\rmP^{+,[n]_2}$.
\end{proposition}

We note that the above is stated as a proposition (and not as a corollary or a remark) since the proof of Theorem~\ref{t:7.7} actually relies on Proposition~\ref{p:1.7e}, which is proved first independently. Another corollary of Theorem~\ref{t:7.7} is given below. Henceforth, we shall use $\rmB_r(x)$ to denote the ball of radius $r \geq 0$ around $x \in \bbT$ in graph-distance.
\begin{corollary}
\label{c:1.11a}
Fix $p \geq 1$ and $\epsilon > 0$. For all $n \geq 1$, let $x_n \in \bbL_n$ and set
\begin{equation}
	r_{n,p} := \frac{n}{p} - (1+\epsilon) \log_2 n \,.
\end{equation} 
Then~\eqref{e:107.41} holds with $A_n := \rmB_{r_{n',p}}(x_{n'}) \cap \bbT_{n'}$.
\end{corollary}
In words, Proposition~\ref{p:1.7e} and Corollary~\ref{c:1.11a} say that under $\rmP_n^+$, the law of $h$ restricted to any finite subset of a subtree rooted at $x_{l_n} \in \bbL_{l_n}$, or a ball of graph radius at most $n'/p - (1+\epsilon) \log_2 n$ around a leaf $x_n \in \bbL_n$, is asymptotically given by the corresponding marginals of $\rmP^{+,[n]_2}$.

\medskip
Having identified the limiting law under $\rmP_n^+$ as $\rmP^{+,[n]_2}$, we wish to study properties of the latter law.  We start with the following proposition, which gives the tails of the marginals of $\rmP^{+,\delta}$. These are inherited 
from the tails of $h$ under $\rmP_n^+$, which were derived in the first paper~\cite[Theorem~\ref{1@t:1.3}]{Work1}.
\begin{proposition}
\label{p:1.11}
There exists $C, c \in (0, \infty)$ such that for all $x \in \bbT$ and $u > 0$,
\begin{equation}
\label{e:101.9a}
c \exp \Big(-C \frac{u^2}{	\ol{\sigma}_\infty(x, u)} \Big)
\leq 
	\rmP^{+,\delta} \big(h(x) > u \big) \leq C \exp \Big(-c \frac{u^2}{	\ol{\sigma}_\infty(x, u)} \Big)
\end{equation}
and 
\begin{equation}
\label{e:101.9b}
c \exp \Big(-C \frac{u^2}{	\ul{\sigma}_\infty(x)} \Big) 
\leq 
	\rmP^{+,\delta} \big(h(x) < -u \big) \leq C \exp \Big(-c \frac{u^2}{	\ul{\sigma}_\infty(x)} \Big) \,,
\end{equation}
where
\begin{equation}
\label{e:101.11a}
	\ul{\sigma}_\infty(x) := |x|+1
	\  , \quad
	\ol{\sigma}_\infty(x, u) := \log_2 u + |x| + 1 \,.
\end{equation}
The lower bound in~\eqref{e:101.9b} only holds up to $u \leq C(m_{|x|}+1)$.
\end{proposition}

Next, we present a key tool in studying various properties of $\rmP^{+,\delta}$, which is of interest by itself. This comes in the form of a coupling result, which shows that the a field $h^{+,\delta}$ on $\bbT$ having law $\rmP^{+,\delta}$ can be realized as a sum of two dependent fields: a standard DGFF $h$ on $\bbT$ and an auxiliary field $\eta^{+,\delta}$, whose gradients decay exponentially in the $\bbL^p$ norm for all $p \geq 1$ away from the root.
\begin{theorem}
\label{t:107.4}
Let $\delta \in [0,1)$. There exists a coupling $\ddot{\rmP}^{+, \delta}$
 between the fields $h$, $h^{+,\delta}$ and $\eta^{+,\delta}$ on $\bbT$ such that 
\begin{enumerate}
	\item $\ddot{\rmP}^{+, \delta}$-almost-surely,
	\begin{equation}
		h^{+,\delta} = h + \eta^{+,\delta} \,.
	\end{equation} 
	\item The laws of $h$ and $h^+$ obey:
	\begin{equation}
		 \ddot{\rmP}^{+, \delta} \big(h \in \cdot) = \rmP(h \in \cdot\,\big|\, h(0) = 0)
		 \ , \quad 
		 \ddot{\rmP}^{+, \delta} \big(h^{+,\delta}\in \cdot) = \rmP^{+,\delta}(h \in \cdot) \,,
	\end{equation}
	and
	\begin{equation}
		 \ddot{\rmP}^{+, \delta} \big(\eta^{+,\delta}(0) \in \cdot\big) = \rmP^{+,\delta}(h(0) \in \cdot) \,.
	\end{equation}
	\item For $x,y \in \bbT$ with $x$ a direct child of $y$, and all $p \geq 1$,
	\begin{equation}
	\label{e:101.39}
		\big\|\eta^{+,\delta}(x) - \eta^{+,\delta}(y) \big\|_p \leq C_p |x|\, 2^{-|x|/p}
		\ , \quad
		\ddot{\rmE}^{+,\delta} \big(\eta^{+,\delta}(x) - \eta^{+,\delta}(y)\big) \geq 0 \,.
	\end{equation}	
	where $C_p < \infty$ and does not depend on $\delta$. 
\end{enumerate} 
\end{theorem}

Thanks to the fast decay of its gradients, the field $\eta^{+,\delta}$ fixates to a random value away from the root, with the asymptotic value depending on the branch followed to infinity. The next theorem quantifies this behavior. For what follows, we denote by ${\bf 1}$ a deterministic field on $\bbT$ which takes the value $1$ at all vertices. 
\begin{theorem}
\label{p:107.5}
Fix $\delta \in [0,1)$ log-dyadic and let $(h, h^{+,\delta}, \eta^{+,\delta})$ be distributed according to $\ddot{\rmP}^{+, \delta}$. For all $x \in \bbT$, the mean of $h^{+,\delta}(x)$ is equal to the mean of $\eta^{+,\delta}(x)$ and both are non-decreasing in $|x|$. If ${\bf x} = (x_n)_{n \geq 1}$ is an infinite branch in $\bbT$ starting at the root, then the limit
\begin{equation}
\label{e:107.37}
	\eta^{+,\delta}({\bf x}) := \lim_{n \to \infty} \eta^{+,\delta}(x_n)
\end{equation}
exists almost-surely and in $\bbL^p$ for all $p \geq 1$, and has a law which does not depend on ${\bf x}$. Moreover, for all $p \geq 1$ there exists $C_p < \infty$, which does not depend on $\delta$, such that for all $A \subset \bbT$ connected, 
\begin{equation}
\label{e:107.38}
	W^{p,\infty} \big(\eta^{+,\delta}_{A},\, \eta^{+,\delta}(\infty), {\bf 1}_{A} \big) \leq C_p \sum_{k \geq 1} \big|A \cap \bbL_k|2^{-k/p} k \,,
\end{equation}
where $\eta^{+,\delta}(\infty)$ has the same law as $\eta^{+,\delta}({\bf x})$. In particular, 
\begin{equation}
\label{e:107.39}
	W^{p,\infty} \big(h^{+,\delta}_{A} ,\, h_{A} + \eta^{+,\delta}(\infty)\, {\bf 1}_{A} \big)
	\leq C_p \sum_{k \geq 1} \big|A \cap \bbL_k|2^{-k/p} k \,,
\end{equation}
where $h^{+,\delta}$ has law $\rmP^{+,\delta}$, and $(h, \eta^{+,\delta}(\infty))$ have some joint law with marginals $\rmP$ and that of $\eta^{+,\delta}(\infty)$ defined above.
\end{theorem}
Thus asymptotic fixation to the same (random) value occurs on a sequence of sets for which the right hand sides of~\eqref{e:107.38} or~\eqref{e:107.39} tend to $0$. This is the case, e.g., at neighborhoods of vertices whose radius decays linearly with the distance from the root, as shown in:
\begin{corollary}
\label{c:1.15a}
Fix $p \geq 1$ and $\epsilon > 0$. For all $n \geq 1$, let $x_n \in \bbL_n$ and set
\begin{equation}
	r_{n,p} := \frac{n}{p} - (1+\epsilon) \log_2 n \,.
\end{equation} 
Then with $A = \rmB_{r_{n',p}}(x_{n'})$ both distances in~\eqref{e:107.38} and~\eqref{e:107.39} tend to $0$ as $n \to \infty$.
\end{corollary}

Combining (the corollaries of) Theorem~\ref{t:7.7} and Theorem~\ref{p:107.5} we get several corollaries, the first of which is,
\begin{corollary}
\label{c:1.15}
Fix $p \geq 1$ and $\epsilon > 0$. For all $n \geq 1$, let $x_n \in \bbL_n$ and set
\begin{equation}
	r_{n,p} := \frac{n}{p} - (1+\epsilon) \log_2 n \,.
\end{equation} 
Then 
\begin{equation}
	W^{p,\infty} \Big(h^+_n\big(\rmB_{r_{n',p}}(x_{n})\big) - m_{n'} \,,\,\,
h_{n'}\big(\rmB_{r_{n',p}}(x_{n'})\big)  + \eta^{+,[n]_2}(\infty) \Big) \underset{n \to \infty} \longrightarrow 0 \,,
\end{equation}
where $h^+_n$ has law $\rmP_n^+$, and $h_{n'}, \eta^{+,[n]_2}(\infty)$ have some joint law with marginals $\rmP_{n'}$ and that of $\eta^{+, [n]_2}(\infty)$ under $\ddot{\rmP}^{+,[n]_2}$ from Theorem~\ref{p:107.5}.
\end{corollary}
In other words, locally at a neighborhood of radius $n'/p-(1+\epsilon)\log_2(n')$ around any leaf $x_n \in \bbL_n$, the law of the DGFF on $\bbT_n$ under the hard wall constraint centered by $-m_{n'}$ is asymptotically the same as that of a standard DGFF on $\bbT_{n'}$ shifted globally by a dependent $\Theta(1)$-random variable $\eta^{+,[n]_2}(\infty)$.

The second corollary gives sharp asymptotics for the mean of $h(x)$ for $x \in \bbT_n$ with $|x| > l_n$ under the conditioning. This improves upon \cite[Corollary~\ref{1@c:1.4}]{Work1} from the first work, in which the asymptotics were given up-to $O(1)$.
\begin{corollary}
\label{c:1.17}	
Let $x_k \in \bbL_k$. For all log-dyadic $\delta \in [0,1)$, the quantity $\rmE^{+,\delta} \big(h(x_k))$ is non-decreasing in $k$ and obeys
\begin{equation}
\label{e:201.50}
	\lim_{k \to \infty} \rmE^{+,\delta} \big(h(x_k)) = \ddot{\rmE}^{+,\delta} \eta(\infty) \,,
\end{equation}
where $\eta(\infty) \equiv \eta^{+,\delta}(\infty)$ is as in Theorem~\ref{p:107.5}
and $x_k \in \bbL_k$. Moreover, for all $k \geq l_n$,
\begin{equation}
\label{e:1.19a}
		\rmE_n^+ h(x_k) = m_{n'} + \rmE^{+,[n]_2} h(x_{k-l_n}) + o(1) \,.
\end{equation}
In particular, 
\begin{equation}
\rmE_n^+ h(x_n) = m_{n'} + \ddot{\rmE}^{+,[n]_2} \eta(\infty) + o(1)  \,,
\end{equation}
and
\begin{equation}
		\rmE_n^+ h(x_{l_n}) = m_{n'} + \rmE^{+,[n]_2} h(0) + o(1) \,.
\end{equation}
All $o(1)$ terms tend to $0$ as $n \to \infty,$ uniformly in $k$.
\end{corollary}

The analogous result for all $p \geq 1$ moments is given by
\begin{corollary}
\label{c:1.18}	
Fix $p \geq 1$ and let $x_k \in \bbL_k$. Then,
\begin{equation}
\label{e:1001.20a}
	\sup_{k,\delta} \Big| \big\|h^{+,\delta}(x_k) \big\|_p - \big\|h(x_k)\big\|_p \Big| < \infty 
\,,
\end{equation}
and
\begin{equation}
\label{e:1001.19a}
\big\|h^+_n(x_k)-m_{n'} \big\|_p = \big\|h^{+,[n]_2}(x_{k-l_n}) \big\| + o(1)\,,
\end{equation}
where $o(1) \to 0$ as $n \to \infty$ uniformly in $x$.
Above $h$, $h^+_n$ and $h^{+,\delta}$ have laws $\rmP_\infty$, $\rmP_n^+$ and $\rmP^{+,\delta}$ respectively.
\end{corollary}
In particular, the $p$-th moment of the height of the field at a vertex in generation $k \geq l_n$ under the conditioning on $\Omega_n^+$ is, up to a uniformly bounded constant, the $p$-th moment of a centered Gaussian with variance $k-l_n$.

\subsection{Results: global statistics}
\label{s:1.3}
Next, we treat several global statistics of interest of the conditional field,
which, among other things, demonstrate the very different nature of the law of $h$ under the conditioning and without. Recall that, for the unconditional field, the empirical population mean is a non-degenerate random quantity (see Lemma~\ref{l:2.4a}). Under the conditioning, it turns out that the population mean tends to a deterministic limit.
\begin{theorem}
\label{t:1.9}
Under $\rmP_n^+$, with $x_n \in \bbL_n$, 
\begin{equation}
\Bigg|\frac{1}{\bbL_n} \sum_{x \in \bbL_n} h(x) - 
\rmE_n^+ h(x_n)\Bigg| \underset{n\to\infty}{\overset{\bbL^1}\longrightarrow} 0 \,,
\end{equation}
In particular, under $\rmP_n^+$,
\begin{equation}
\frac{1}{\bbL_n} \sum_{x \in \bbL_n} h(x) = 
m_{n'}  + \rmE^{+,[n]_2} \eta(\infty) + o(1)\,,
\end{equation}
where $o(1) \to 0$ in probability as $n \to \infty$ and $\eta(\infty) \equiv \eta^{+,\delta}(\infty)$ is as Theorem~\ref{p:107.5}.
\end{theorem}

Next we turn to the so-called additive (exponential) martingales associated with our BRW $h$. It is well known that in the unconstrained case, whenever $\alpha \in [0,c_0)$, 
\begin{equation}
\label{e:101.50}
\frac{1}{|\bbL_n|} \sum_{x \in \bbL_n} 
\exp\Big(\alpha h(x) - \frac{\alpha^2}{2}n\Big) \underset{n \to \infty}\longrightarrow Z_\alpha \,,
\end{equation}
both almost-surely and in $\bbL^1$, where $Z_\alpha \in (0,\infty)$ is a non-degenerate random variable. The almost-sure convergence is an immediate  consequence of the martingale convergence theorem, as the processes on the left hand side is a non-negative martingale. Showing $\bbL^1$ convergence and the non-degeneracy of the limit requires non-trivial work.
The following theorem shows that under $\rmP_n^+$, with the scaling properly modified, the above process, which is no longer a martingale, still converges in $\bbL^1$ (and hence in probability), albeit to a deterministic value.
\begin{theorem}
\label{t:1.8}
Let $\alpha \in \big[0, c_0]$. For all log-dyadic $\delta \in [0,1)$ the limit
\begin{equation}
A^\delta_\alpha := \lim_{|x| \to \infty} \rmE^{+, \delta}  \exp\Big(\alpha h(x)-\frac{\alpha^2}{2}|x|\Big) \,,
\end{equation}
exists and is bounded uniformly in $\delta$ and $\alpha$. Moreover, for all $\alpha \in [0,c_0)$, under $\rmP_n^+$,
\begin{equation}
\Bigg|
\frac{1}{|\bbL_n|} \sum_{x \in \bbL_n} \exp \Big(\alpha h(x) -\alpha m_{n'}-\frac{\alpha^2}{2} n'\Big) - A_\alpha^{[n]_2}\Bigg| \underset{n \to \infty}{\overset{\bbL^1}\longrightarrow} 0 \,.
\end{equation}
\end{theorem}
We remark that while our proofs in fact show that the first limit holds for all $\alpha>0$, for the second we need the uniform integrability of the martingale in~\eqref{e:101.50}, which is known to hold only up to (and excluding) the critical $\alpha = c_0$. We leave the more careful treatment of the critical and super-critical exponential martingales to a future work.

Turning to extreme values, we recall that in the unconditional case, 
\begin{equation}
\label{e:01.13}
-\min_{x \in \bbL_n} h(x)\, + m_n
\,\underset{n \to \infty}{\overset{\rmd} \longrightarrow}\,
	G \oplus c_0^{-1} \log Z 
\ , \quad 
\max_{x \in \bbL_n} h(x)\, - m_n
\,\underset{n \to \infty}{\overset{\rmd} \longrightarrow}\,
	G \oplus c_0^{-1} \log Z \,,
\end{equation}
where $G$ has a Gumbel law with rate $c_0$ and $Z \in (0,\infty)$ is some non-trivial random variable which is independent of $G$ (see \cite{AidekonBRW13}). As in~\cite{Work1}, we write $\oplus$ to for the sum of two independent random variables. The following theorem shows that under the hard-wall constraint, the maximum, centered properly, still converges to the same Gumbel, but without a random shift.
\begin{theorem}\label{t:1.5}
There exists an explicit constant $c_1 \in \bbR$ such that with
\begin{equation}
\label{e:1.25}
	m_n^+ := 2m_{n'} + c_0^{-1} \log_2 n + c_0^{-1} \log\log_2 n + c_0^{-1} \log A_{c_0}^{[n]_2} + c_1 \,,
\end{equation}
where $A_{c_0}^{\delta}$ is as in Theorem~\ref{t:1.8} and $c_0$ is as in~\eqref{Lisa.1},
it holds for all $u \in \bbR$ that
\begin{equation}
\rmP_n^+ \Big(\max_{x \in \bbL_n} h(x) - m_n^+ \leq u \Big) 
\underset{n \to \infty}  \longrightarrow \rme^{-\rme^{-c_0 u}} \,.
\end{equation}
In particular, under $\rmP_n^+$,
\begin{equation}
	\max_{x \in \bbL_n} h(x) - m_n^+  \, \underset{n \to \infty}{\overset{\rmd}{\longrightarrow}}\, \rm{Gumbel}(c_0) \,.
\end{equation}
\end{theorem}
Since $A_{c_0}^{[n]_2}$ is bounded, we get that the centering sequence for the constrained maximum satisfies
\begin{equation}
	m_{n'} + m_{n'} \ll m_n^+ \ll m_{n'} + m_n \,,
\end{equation}
with both bounds above being a-priori na\"ive guesses for this quantity. 

On the other hand, for the minimum under the conditioning, we have
\begin{theorem}
\label{thm:minimum}
For all log-dyadic $\delta \in [0,1)$, with $p_\infty$ is as in~\eqref{e:1.1a}, set
\begin{equation}
	\kappa_{\delta} := - 2^{-\delta} \rmE^{+,\delta} \bigg(\frac{\rmd}{\rmd s} \log p_{\infty}(s)\big|_{s=-h(0)}\bigg) \,,
\end{equation}
which is uniformly bounded and positive. Then, for all $u \geq 0$,
\begin{equation}
\label{e:1.28}
\rmP_n^+\big(n \kappa_{[n]_2} \min_{x\in \bbL_n} h(x)\leq u \big) \underset{n \to \infty}\longrightarrow 1-\rme^{-u} \,.
\end{equation}
In particular, under $\rmP_n^+$,
\begin{equation}
	n \kappa_{[n]_2} \min_{x\in \bbL_n} h(x) \,\underset{n \to \infty}\Longrightarrow\, \rm{Exp}(1) \,.
\end{equation}
\end{theorem}

Lastly, we state what by now should be obvious.
\begin{corollary}
\label{c:1.11}
For any $(\mu_n)_{n \geq 1}$, with $\mu_n \in \bbR^{\bbL_n}$, the law of $h$ under $\rmP_n$ and the law of $h + \mu_n$ under $\rmP_n^+$ are asymptotically mutually singular with respect to each other. That is, there exists a sequence $(A_n)_{n \geq 1}$ of measurable subsets of $\bbR^{\bbT_n}$ such that
\begin{equation}
	\lim_{n \to \infty} \rmP_n \big(h \in A_n \big) = 1
	\quad , \qquad
		\lim_{n \to \infty} \rmP^+_n \big(h + \mu_n \in A_n \big) = 0 \,.
\end{equation}
\end{corollary}

\subsection{Proof overview}
\label{s:1.4}
Let us now give a brief overview of the arguments leading to the results in this paper. 
\subsubsection{Asymptotic law at generation $l_n$}
The first step in obtaining asymptotics for both local and global observables, is the derivation of the asymptotic law of 
\begin{equation}
\hat{h}(x) \equiv h(x) - \mu_n(x) \,.
\end{equation}
The crucial ingredient here is a combination of two results from~\cite{Work1} (Proposition~\ref{1@p:4.2} and Proposition~\ref{1@p:4.2a}) which, for self containment, appears here as Proposition~\ref{p:106.3}. Setting henceforth $x_k := [x]_k$ for some $x \in \bbL_n$, it reads:
\begin{equation}
\label{e:1001.63}
		\Big\| \rmP^+_n \big(\hat{h}(x_{k'}) \in \cdot \,\big|\, \hat{h}(x_k) = v'\big)  - 
					\rmP^+_n \big(\hat{h}(x_{k'}) \in \cdot \,\big|\, \hat{h}(x_k) = v\big) \Big\|_{\rm TV}  \leq C\rme^{-(c(k'-k)-|v|\vee |v'|)^+} \,,
\end{equation}
for $0 \leq k \leq k' \leq l_n$, any $v,v' \in \bbR$. This was derived in~\cite{Work1} using the representation of $(\hat{h}([x]_k) :\: k \leq l_n)$ under the conditional law, as a random walk which is subject to a localizing force that attracts it to zero (see Section~\ref{1@s:6} in~\cite{Work1} and also Subsection~\ref{1@s:1.3} therein for further explanations).

Taking $l_n$ and $l_n - k$ in place of $k'$ and $k$ in~\eqref{e:1001.63} respectively, and using also the tightness of $\hat{h}(x_j)$ for $j \leq l_n$ (Theorem~\ref{1@t:1.3} in \cite{Work1}), this 
shows that the law of $\hat{h}(x_{l_n})$ is, asymptotically as $k \to \infty$, independent of the value of $\hat{h}(x_{l_n-k})$. Using the total probability formula and tightness again, we thus get
\begin{equation}
\begin{split}
	\rmP^+_n \Big(\hat{h}(x_{l_n}) \in \cdot \Big) & =  \rmP^+_n \Big(\hat{h}(x_{l_n}) \in \cdot \,\Big|\, \hat{h}(x_{l_n-k}) = 0 \Big) + o(1) \\ 
& =	\rmP_{n-l_n+k}\Big(h(x_k) - c_0 2^{[n]_2} 2^{k} \in \cdot 
\,\Big|\, \Omega_{n-l_n+k}\big(c_0 \big(2^{[n]_2} 2^{k} + k\big) \big) \Big) + o(1) \\
& =	\rmP_{n-l_n+k, -c_0 2^{[n]_2} 2^{k}}\Big(h(x_k) \in \cdot \,\Big|\, \Omega_{n-l_n+k}(c_0 k) \Big) + o(1) \\
& = \rmP^{\uparrow c_0 k}_{-c_0 2^{[n]_2} 2^{k}} \big(h(x_k) \in \cdot \big) + o(1) \,,
\end{split}
\end{equation}
where the second equality follows using the Markov property of $h$ and some algebra, the third by shifting the entire field by $- c_0 2^{[n]_2} 2^{k}$ and the forth thanks to Proposition~\ref{p:1.45}. Above the addition of the $o(1)$ terms should be interpreted as meaning that the $W^{p,\infty}$ distance between the two sides of the equality sign tends to $0$ as $n \to \infty,$ followed by $k \to \infty$. We remark that Proposition~\ref{p:1.45} follows easily from the weak convergence of the centered minimum and the spatial-Markov property of the DGFF and its proof is a rather standard $h$-transform type argument.

Since the probability on the left most side above does not depend on $k$, while the one on the right most side depends on $n$ only through $[n]_2$, standard arguments (see Lemma~\ref{l:101.7}) show both the existence of the $W^{p,\infty}$-limit
\begin{equation}
\rmP^{+,\delta}\big(h(0) \in \cdot \big) :=  \lim_{k \to \infty} \rmP^{\uparrow c_0 k}_{-c_0 2^\delta 2^k} \big(h(x_k) \in \cdot\big) \,,
\end{equation}
and the $W^{p,\infty}$-convergence 
\begin{equation}
\label{e:200.65}
		W^{p,\infty} \Big(\rmP^+_n\big(\hat{h}(x_{l_n}) \in \cdot\big)
\,,\,\, \rmP^{+,[n]_2} \big(h(0) \in \cdot \big)\Big) \underset{n \to \infty}{\longrightarrow}
	0 \,,
\end{equation}
which is the first part of Theorem~\ref{t:1.7} and Proposition~\ref{p:1.7e} restricted to vertices in generation $l_n$.

\subsubsection{Local statistics}
To extend these results beyond generation $l_n$, we can use the Markov property to write
\begin{equation}
\label{e:400.67}
	\rmP_n^+\big(\hat{h}_{\bbT_{n'}(x_{l_n})} \in \cdot \big) 
	= \int \rmP^{\uparrow 0}_{n', v}\big(\cdot \big) \, \rmP_{n}^+ \big(\hat{h}(x_{l_n}) \in \rmd v\big) \,.
\end{equation}
Thanks to~\eqref{e:200.65} and Proposition~\ref{p:1.45} one can then show that (Proposition~\ref{p:1.7e})
\begin{equation}
\label{e:401.68}
		W^{p,\infty}_{\loc} \Big(\rmP^+_n \big(\hat{h}_{\bbT_{n'}(x_{l_n})} \in \cdot\big)
\,,\,\, \rmP^{+,[n]_2}\circ \Pi_{\bbT_{n'}}^{-1} \Big) \underset{n \to \infty}{\longrightarrow}
	0 \,,
\end{equation}
where 
\begin{equation}
\label{e:200.67}
\rmP^{+,\delta}\big(\cdot\big) := \int \rmP^{\uparrow 0}_{\infty, v}\big(\cdot \big) \, \rmP^{+,\delta} \big(h(0) \in \rmd v\big) 
	 \,.
\end{equation}

In view of the last two displays, to derive properties of $\rmP^{+,\delta}$ and strengthen the topology in~\eqref{e:401.68}, one needs to study the law $\rmP^{\uparrow u}_{n, v}$ for $n \in \bbN \cup \{\infty\}$ and $u,v \in \bbR$. To this end, 
we introduce in Proposition~\ref{p:7.9} a Markovian coupling between the field $h^{\uparrow u}_{n,v}$ which has law $\rmP^{\uparrow u}_{n, v}$ and the field $h$, which has law $\rmP_{n,0},$ such that the difference
\begin{equation}
\label{e:300.69}
\eta^{\uparrow u}_{n,v} :=	h^{\uparrow u}_{n,v} - h 
\end{equation} 
satisfies 
\begin{equation}
\label{e:200.69}
\big\|\eta^{\uparrow u}_{n,v}(x) - \eta^{\uparrow u}_{n,v}(y) \big\|_p \leq 
		C_p \big(((u-v)^-)^{1/p} + 1\big) \rme^{c_0 (u-v)/p}\,|x|\, 2^{-|x|/p} \,,
\end{equation}
whenever $x,y \in \bbT_n$ are nearest neighbors.

Since this coupling is Markovian, it induces via~\eqref{e:200.67}  a similar coupling between $h^{+,\delta}$ with law $\rmP^{+,\delta}$ and $h$ with law $\rmP_{\infty, 0}$ (Theorem~\ref{t:107.4}). In particular, thanks to the ``almost-Gaussian'' tails of $h(0)$ under $\rmP^{+,\delta}$ which are inherited via~\eqref{e:200.65} from those of $\hat{h}(x_{l_n})$ under $\rmP_n^+$ (Theorem~\ref{1@t:1.3} in~\cite{Work1}), the exponential bound in~\eqref{e:200.69} translates into
\begin{equation}
\big\|\eta^{+,\delta}(x) - \eta^{+,\delta}(y) \big\|_p \leq C_p |x|\, 2^{-|x|/p} \,,
\end{equation}	
where
\begin{equation}
\label{e:200.70}
	\eta^{+,\delta} = h^{+,\delta} - h \,.
\end{equation}
It is not difficult to show that the above exponentially fast vanishing of the gradient of $\eta^{+,\delta}$ at large depths, translates into fixation of $\eta^{+,\delta}$ on connected sets far away from the root and, through~\eqref{e:200.70}, to $h^{+,\delta}$ on such sets being a random global shift of the usual DGFF $h$ (Theorem~\ref{p:107.5}). 

Moreover, since the field $h$ has the same law under the coupling 
between $\rmP^{\uparrow u}_{n,v}$ and $\rmP_{n,0}$ in~\eqref{e:300.69}
 for all $n \in \bbN \cup \{\infty\}$, and since the bound~\eqref{e:200.69} holds with uniform constants for all such $n$, one can derive yet another coupling between  $h^{\uparrow u}_{n,v}$ with law $\rmP^{\uparrow u}_{n,v}$ and $h^{\uparrow u}_{\infty, v}$ with law $\rmP^{\uparrow u}_{\infty,v} \circ \Pi_{\bbT_n}^{-1}$ such that the  difference field: $\eta^{\uparrow u}_{n,\infty,v} := h^{\uparrow u}_{\infty,v} - h^{\uparrow u}_{n,v}$ is bounded as in~\eqref{e:200.69}, uniformly in $n$. This, in turn, allows us to strengthen the topology for the convergence of $\rmP^{\uparrow u}_{n,v}$ to $\rmP^{\uparrow u}_{\infty,v}$ and, through~\eqref{e:400.67}, for the asymptotic equivalence between between laws 	$\rmP_n^+\big(\hat{h}_{\bbT(x_{l_n})} \in \cdot \big)$ and $\rmP^{+,[n]_2}$, thus yielding Theorem~\ref{t:7.7}.

\subsubsection{Global statistics}
Turning to global observables, as discussed in Subsection~\ref{s:1.1}, the law of $\hat{h}$ on $\bbT_n \setminus \bbT_{l_n}$ under $\rmP_n^+$ is that of a union of $2^{l_n} \in (n/2, n]$
subfields on the subtrees of depth $n' = n -l_n$ rooted at generation $l_n$, each having law which is asymptotically equivalent to $\rmP^{+,[n]_2} \circ \Pi_{\bbT_{n'}}^{-1}$ and they are ``almost independent'' of each other, in the sense that correlations between subfields decay exponentially fast in the graph distance between the roots of their respective domains.

This asymptotic picture suggests, via a usual weak-law-of-large-numbers (WLLN) reasoning, that global observables of the field, which can be expressed as an average (over all subfields) of quantities which are measurable w.r.t. each subfield and tend to an asymptotic law, converge weakly to a {\em deterministic} limit as $n \to \infty$. We remark that in the unconditional case, different subfields are still identically distributed, but even if their domains are at graph-distance $\Theta(l_n)$ apart, they remain non-trivially correlated through the value of $h$ in the first $l_n$ generations. As such, global observables of the form above still admit a weak limit, but the limit turns out to be {\em non-deterministic}. This contrast is reflected in the comparison between the limiting statements in both the conditional and unconditional case, which is drawn following each result in Subsection~\ref{s:1.3}.

To streamline a formal derivation of the above WLLN argument for the various global observables considered, we rely on two key results. The first (Proposition~\ref{p:4.5b}) shows convergence of the averaged quantities to their mean, without estimating the mean itself. It states that if $(\xi_n(x_{l_n}) :\: x_{l_n} \in \bbL_{l_n},\, n \geq 1)$ is a uniformly integrable collection of random variables such that $\xi_n(x_{l_n})$ is measurable w.r.t. $\sigma(h_{\bbT_{n'}(x_{l_n})})$ and has conditional law $\rmP_n^+(\xi_n(x_{l_n}) \in \cdot\, | \,h(x_{l_n}))$ which does not depend on $x_{l_n}$, then 
\begin{equation}
\label{e:205.5}
\bigg| \frac{1}{|\bbL_{l_n}|} \sum_{x_{l_n} \in \bbL_{l_n}} \xi_n(x_{l_n})
	- \rmE_n^+ \xi_n(x_{l_n}) \bigg|
\,	{\underset{n \to \infty}\longrightarrow} \, 0 \,.
\end{equation}
both in $\rmP_n^+$-probability and in $\bbL^1$.

To replace the mean $\rmE_n^+ \xi_n(x_{l_n})$ with an explicit, asymptotically equivalent, quantity, we rely on the second key result (Lemma~\ref{l:103.10}), which was in fact proved in~\cite{Work1} (Lemma~\ref{1@l:103.10}). It says that if $\hat{\xi}_{n'}$ is measurable w.r.t. $\sigma(h_{\bbT_{n'}})$, then for all $m \geq 0$, $u \in \bbR$,
	\begin{equation}
		\label{e:303.52}
		\Big|\rmE_{n'}^{\uparrow u} \hat{\xi}_{n'} - \rmE_{n'}^{\uparrow u} \hat{\xi}_m \Big| 
		\leq C\big(\rme^{C u^2} + 1\big) \sum_{k=m}^{n'}
		k \rme^{-c_0 k} \rmE_{n'} \Big((h(x_k)^++1) \rme^{-c_0 h(x_k)} |\hat{\xi}_k|\Big)\,.
\end{equation}
where $\hat{\xi}_k := \rmE_{n'} \Big(\hat{\xi}_{n'}\,\big|\, h\big(\{x_k\} \cup (\bbT_{n'} \setminus \bbT_{n'-k}(x_k)) \big)\Big)$ with $x_k = [x]_k$ as before. (Here and after, we view $\xi_n(x_{l_n})$, $\hat{\xi}_{k}$ as measurable functions on $\bbR^{\bbT_{n'}}$, $\bbR^{\bbT_k}$.)

Since the definition of $\hat{\xi}_k$ and the right hand side in~\eqref{e:303.52} involve expectation w.r.t. the law $\rmP_{n'}$ of the unconditional field, they are tractable quantities. At the same time, if the above sum vanishes as $n' \to \infty$ followed by $m \to \infty$ then~\eqref{e:303.52} implies that
\begin{equation}
\label{e:2000.76}
\rmE_{n'}^{\uparrow u} \hat{\xi}_{n'} = \rmE_{n'}^{\uparrow u} \hat{\xi}_m + o(1) \,,	
\end{equation}
with $o(1) \to 0$ in the same limits of $n'$ and $m$. If it also happens that $\hat{\xi}_{n'}$ is measurable w.r.t. $\sigma(h_{\bbT_{n'-m}(x_m)})$ (or can be  well approximated in $\bbL^1$ by such a random variable), then $\hat{\xi}_m$ is measurable w.r.t. $\sigma(h(x_m))$ by the Markov property, and we can replace
$\rmE_{n'}^{\uparrow u} \hat{\xi}_m$ by $\rmE_\infty^{\uparrow u} \hat{\xi}_m$ in~\eqref{e:2000.76}, thanks to Proposition~\ref{p:1.45}. This implies, via standard arguments that
\begin{equation}
\label{e:2177}
	\lim_{n' \to \infty} \rmE_{n'}^{\uparrow u} \hat{\xi}_{n'} = \lim_{m \to \infty} \rmE_{\infty}^{\uparrow u} \hat{\xi}_m =: \hat{A}(u) \in (-\infty, \infty) \,.
\end{equation}
See, e.g., Propositions~\ref{l:2.8b} and~\ref{p:2.7a}.

Finally, we take $\hat{\xi}_{n'} = \xi_n(x_{l_n}) \circ \tau_{m_{n'}+v}$ in the above derivation, with $v \in \bbR$ and $\tau_w :\bbR^{\bbT_{n'}} \to \bbR^{\bbT_{n'}}$ being the translation mapping $(\tau_w h)(x) = h(x) + w$. Then, letting $\hat{A}(u; v)$ denote the corresponding limit in~\eqref{e:2177} for $u \in \bbR$, we thus obtain thanks to~\eqref{e:401.68} and \eqref{e:200.67}, 
\begin{equation}
\rmE_n^+ \xi_n(x_{l_n}) = \rmE^{+,[n]_2} \hat{A} \big(-h(0);\; h(0)\big) + o(1) \,.
\end{equation}
where $o(1) \to 0$ as $n \to \infty$. This gives the desired asymptotics for the mean in~\eqref{e:205.5}. 

\subsubsection*{Paper outline}
The remainder of the paper is organized as follows. Section~\ref{s:2} contains known results about the unconditional field. Section~\ref{s:3} includes the proofs of all statements concerning local statistics of the conditional field. Finally, Section~\ref{s:4} treats global statistics thereof.

\section{DGFF and branching random walk preliminaries}
\label{s:2}
In this sub-section we collect basic results concerning the unconditional field, viewed both as a BRW and as a DGFF. These results are either standard, easily proven or already derived in the first work in this series.

\subsection{FKG property and consequences}
\label{ss:2.1}
\begin{lemma}[Lemma~\ref{1@lemma:FKG} in~\cite{Work1}]
\label{lemma:FKG}
	For all $u,v \in \bbR$, the law 
	$\rmP_{n,v}^{\uparrow u}$ is FKG. In particular, $\rmP_n^+$ is FKG.
\end{lemma}

As a consequence of FKG we have the following,
\begin{lemma}[Lemma~\ref{1@l:2.4} in~\cite{Work1}]
	\label{l:2.4}
	Let $\varphi: \bbR^{\bbT_n} \to \bbR_+$ be continuous, non-decreasing and positive. Then for all $u \in \bbR$, the function
	\begin{equation}
		v \mapsto \rmE_{n,v}^{\uparrow u} \varphi(h)
	\end{equation}
	is non-decreasing in $v \in \bbR$.
\end{lemma}

\subsection{Tail estimates for the minimum}
A-priori estimates on the right tail are given by the following lemma. Recall that $p_n$, $p_\infty$ are the right tail function and its limit, as defined in~\eqref{e:3.3} and~\eqref{e:1.1a}.
\begin{lemma}[Lemma~\ref{1@l:4.2a} in~\cite{Work1}]
	\label{l:4.2a}
	There exist $C \in (0,\infty)$ such that for all $n \geq 1$ and $u \in \bbR$,
	\begin{equation}
		\label{e:4.10}
		-\frac{1}{1-2^{-n}} u^+ - C \leq \frac{\rmd}{\rmd u} \log  p_n(u) \leq -\frac{1}{1-2^{-n}} u^+ + C \log (u\vee \rme) .
	\end{equation}
	In particular,
	\begin{equation}
		c \exp \Big(-\tfrac{1}{2-2^{-n+1}} (u^+)^2\Big)
		\leq p_n(u) \leq  C \exp \Big(-\tfrac{1}{2-2^{-n+1}} (u^+)^2 + C u \log (u\vee \rme)\Big) \,.
	\end{equation}
\end{lemma}

\begin{lemma}[Lemma~\ref{1@l:2.9a} in~\cite{Work1}]
\label{l:2.9a}
	It holds that 
	\begin{equation}
		\label{e:2.33}
		\lim\limits_{n\to \infty}\frac{\mathrm{d}}{\mathrm{d}u} \log p_n(u)= \frac{\mathrm{d}}{\mathrm{d}u} \log p_\infty(u)
	\end{equation}
	uniformly in $u$ on compact sets. For all $n \in [0, \infty]$, the derivatives 
	$\frac{\mathrm{d}}{\mathrm{d}u} \log p_n(u)$ are continuous and satisfy
	\begin{equation}
		\label{e:2.37b}
		-2u^+ - C \leq \frac{\mathrm{d}}{\mathrm{d}u} \log p_n(u) \leq 0 \,.
	\end{equation}
\end{lemma}

Turning to the left tail:
\begin{equation}
	\label{e:103.24}
	q_n(u) := 1 - p_n(-u) = \rmP_n \Big(\min_{\bbL_n} h \leq -m_n - u \Big) \,,
\end{equation}
we have
\begin{lemma}[\cite{bramsondingoferbrw}, Proposition~3.1]
	\label{l:2.10}
	There exist $C,c \in (0,\infty)$ such that for all $n \geq 1$, $u \geq 0$,
	\begin{equation}
		q_n(u)
		\leq C(u+1) \rme^{-c_0 u}  \,, 
	\end{equation}	
	and for all $n \geq 1$, $u \in [0, n^{1/2}]$,
	\begin{equation}
		q_n(u)
		\geq c (u+1) \rme^{-c_0 u} \,.
	\end{equation}	
	Moreover, there exists $C_0 \in (0,\infty)$ such that
	\begin{equation}
		\lim_{u \to \infty} \limsup_{n \to \infty}
		\Bigg|\frac{q_n(u)}{C_0 u \rme^{-c_0 u}} - 1 \Bigg| = 0 \,.
	\end{equation}	
\end{lemma}

\subsection{Additive martingales}
Next, we recall that the (additive) exponential martingales associated with the DGFF on $\bbT_n$ are the processes $Z_\alpha = (Z_{n,\alpha})_{n \geq 1}$ defined for $\alpha > 0$ by
\begin{equation}
	Z_{n,\alpha} := \frac{1}{|\bbL_n|} \sum_{x \in \bbL_n}
	\rme^{\alpha h(x) - \frac{n}{2} \alpha^2} \,.
\end{equation}
As positive martingales, they converge almost-surely. The fact that the convergence is in $\bbL^p$ for some $p > 1$ whenever $\alpha < \sqrt{2\log 2}$,
and consequently the non-vanishing of the limit, is non-obvious and can be found, e.g., in~\cite{biggins902}:
\begin{proposition}
\label{p:2.6} 
For all $\alpha < \sqrt{2 \log 2}$, the limit
\begin{equation}
	Z_\alpha := \lim_{n \to \infty} Z_{n,\alpha} \,,
\end{equation}
exists almost-surely and in $\bbL^p$ for $p \in \big[1, (2 \log 2)/\alpha^2 \wedge 2\big)$. 
\end{proposition}
\begin{proof}
By~\cite[Theorem~1]{biggins902} with $\lambda = \theta = \alpha$, $\alpha := p$ and $\gamma = 2$, the limit in the statement of the proposition is in $\bbL^p$ for all $p \in (1,2]$ such that
\begin{equation}
\frac12 \alpha^2 p^2 + \log 2 - \frac12 \alpha^2 p - p \log 2 < 0 \,,
\end{equation}
which is equivalent to $p < (2 \log 2)/\alpha^2$.
\end{proof}

Lastly, the next lemma follows by a simple computation.
\begin{lemma}
\label{l:2.4a}
For any $n\geq 0$,
\begin{equation}\label{eq:2.15}
	\Var_n \bigg(\frac{1}{|\bbL_n|} \sum_{x \in \bbL_n} h(x)\bigg) = 1-2^{-n}.
\end{equation}
\end{lemma}
\begin{proof}
Since $\rmE_n(h(x))=0$ all $x\in \bbT_n$,
\begin{multline}\label{eq:2.16}
	\Var_n \bigg(\frac{1}{|\bbL_n|} \sum_{x \in \bbL_n} h(x)\bigg) =\rmE_n\bigg[ \bigg(\frac{1}{|\bbL_n|} \sum_{x \in \bbL_n} h(x)\bigg)^2 \bigg]\\
	= | \bbL_n|^{-2} \rmE_n\bigg[ \sum_{x \in \bbL_n}h(x)^2+ \sum_{x \in \bbL_n}\sum_{k=1}^n \sum_{y \in \bbL_n: |x\wedge y | =n-k} h(x) h(y) \bigg].
\end{multline}
Now, $|\bbL_n| =2^n$ and for each $x\in \bbL_n$ and $k=1,\dotsc, n$ there exist $2^{k-1}$ leaves $y\in \bbL_n$ whose most recent common ancestor, $x\wedge y$ is at depth $k$. Also $\rmE_n \big( h(x)h(y)\big) = n-k,$ for any $x,y\in \bbL_n$ such that $|x\wedge y|=n-k$. Therefore and by linearity of expectations \eqref{eq:2.16} is equal to
\begin{equation}
	2^{-2n} \bigg( 2^n n + 2^n\sum_{k=1}^n 2^{k-1}(n-k)   \bigg)=2^{-2n}\big(2^{n}n +2^{2n} -2^n n-2^n \big)=1 - 2^{-n},
\end{equation}
which shows \eqref{eq:2.15}.
\end{proof}

\section{Local statistics}
\label{s:3}
\subsection{Infinite volume limit under typical conditioning on the minimum}
We start by showing that $\rmP_{n,v}^{\uparrow u}$ admits an infinite volume limit, namely proving Proposition~\ref{p:1.45}. Given a probability measure $\rmP$ on $\bbR^{\bbV}$ for some finite non-empty set $\bbV$, and a function $\varphi : \bbR^\bbV \to (0, \infty)$, we define $\rmP^\varphi$ via 
\begin{equation}
\label{e:2.10a}
	\frac{\rmd \rmP^\varphi}{\rmd \rmP}  = \Big(\textstyle \int \varphi \rmd \rmP \Big)^{-1} \varphi \,.
\end{equation} 
Setting for all $r \geq 0$ and $u \in \bbR$, 
\begin{equation}
\label{e:5.71}
\varphi_{\infty,r,u}(h) = \prod_{x \in \bbL_r} p_{\infty} \big(u - c_0 r - h(x)) \,,
\end{equation}
Equation~\eqref{e:105.70} in Proposition\ref{p:1.45} can now be equivalently written as
\begin{equation}
\label{e:100.160}
\rmP_{\infty, v}^{\uparrow u} \circ \Pi_{\bbT_{r}}^{-1}=\rmP_{r,v}^{\varphi_{\infty, r,u}} \,,
\end{equation}
for all $v \in \bbR$. 

The proposition will follow directly, once we show the following lemma.
\begin{lemma}
\label{l:5.7}
For all $r \geq 0$, the limit
\begin{equation}
\label{e:5.70}
	\rmE_n^{\uparrow u} F(h_{\bbT_r} + v)
	\underset{n \to \infty} \longrightarrow \rmE_r^{\varphi_{\infty,r,u}} F(h_{\bbT_r} + v)\,,
\end{equation}
holds uniformly in $u,v \in \bbR$ on any compact subset of $\bbR$ and all $F : \bbR^{\bbT_r} \to \bbR$ satisfying 
\begin{equation}
\label{e:107.2}
	F(h) \leq C \rme^{C\|h\|}\,,
\end{equation}
for some arbitrary (but fixed) $C < \infty$. Moreover, the right hand side of~\eqref{e:5.70}
is continuous in $u,v \in \bbR$.
\end{lemma}
\begin{proof}
As in \eqref{1@e:2.32d} of \cite{Work1}, for $0 \leq k \leq n$ and $u \in \bbR$ we may write
\begin{equation}
\label{e:2.32d}
	\rmP_n \big(h_{{\bbT_k}} \in \cdot \,\big|\, \Omega_n(u)\big) = \rmP_k^{\varphi_{n,k,u}}(\cdot) \,,
\end{equation}
where
\begin{equation}
\label{e:2.40d}
\varphi_{n,k,u}(h) := \prod_{x \in \bbL_k} p_{n-k} \big(-h(x) + u - m_n + m_{n-k} \big) \,.
\end{equation}
Hence, the difference between the two sides in~\eqref{e:5.70} is at most
\begin{equation}
\label{e:107.4}
	\int_{h \in \bbR^{\bbT_r}} \big|F(h)\big|\, \big|\varphi_{n-r,r,u+O(r/n)}(h-v)
	- \varphi_{\infty,r,u}(h-v)\big| \,\frac{\rmP_r(\rmd h)}{\rmd h}\, (h-v) \,\rmd h \,,
\end{equation}
where $\rmP_r(\rmd h)/\rmd h$ is the density of $\rmP_r$ w.r.t. Lebesgue on $\bbR^{\bbT_r}$.
Thanks to~\eqref{e:107.2}, the boundedness of $\varphi_{n,k,u}$ and the Gaussian tails of
$\rmP_r(\rmd h)/\rmd h$, we may restrict the above integral in $h$ to a compact subset of $\bbR^{\bbL_r}$ at an arbitrarily small cost, uniformly in $v$ and $F$ as desired.
Then the uniform convergence of $p_n$ to $p_\infty$ on compact subsets of $\bbR$ (which follows from the continuity of the limit, being the tail of an absolutely continuous random variable, and the monotonicity of $p_n$) and continuity of $p_\infty$ implies that the restricted integral can be made arbitrarily small, uniformly as desired.

To show continuity in $u,v$ of the limit, we write the latter as~\eqref{e:107.4}, with the terms in absolute value replaced by $F(h) \varphi_{\infty, r, u}(h-v)$ respectively. Then, continuity of the integrand in $u$, $v$ follows from the continuity of both $p_\infty$ and the Gaussian density in their respective arguments. The result is then a consequence of the Dominated Convergence Theorem.
\end{proof}

With the above lemma we readily have,
\begin{proof}[Proof of Proposition~\ref{p:1.45}]
Thanks to Lemma~\ref{l:5.7} with $v=0$, for all $r \geq 1$ and $p \geq 1$,
\begin{equation}
	\int \|h_{{\bbT_r}}\|_p \big| \big(\rmP_n^{\uparrow u} \circ \Pi^{-1}_{\bbT_r}\big) (\rmd h) -  \rmP_r^{\varphi_{\infty, r,u}}(\rmd h)\big| \underset{n \to \infty}\longrightarrow 0\,,
\end{equation}
which implies by, e.g., Theorem 6.15 in~\cite{villani} that,
\begin{equation}
\label{e:107.6}
	W^{p,\infty} \Big(\rmP_n^{\uparrow u} \circ \Pi^{-1}_{\bbT_r},\, \rmP_r^{\varphi_{\infty, r,u}} \Big) \underset{n \to \infty}\longrightarrow 0 \,.
\end{equation}
Consistency of the family $\{\rmP_n^{\uparrow u} \circ \Pi^{-1}_{\bbT_r} :\: r \geq 1\}$ for all $n \geq 1$ is then inherited in the limit and so by Kolmogorov's Extension Theorem we may define a unique law $\rmP^{\uparrow u}_\infty$ on $\bbR^{\bbT},$ whose marginal on $\bbR^{\bbT_r}$ is $\rmP_r^{\varphi_{\infty, r,u}}$ for all $r \geq 1$. From~\eqref{e:107.6} for all $r \geq 1$, we conclude that $\rmP^{\uparrow u}_\infty$ is the $W_\loc^{p, \infty}$-limit of $\rmP_n^{\uparrow u}$ as $n \to \infty$ for all $p \geq 1$. Since $\rmP_{n,v}^{\uparrow u} \equiv \rmP_{n}^{\uparrow u-v} (\cdot - v)$, setting $\rmP^{\uparrow u}_{\infty, v}$
via~\eqref{e:107.28}, we thus get~\eqref{e:5.69},~\eqref{e:107.28} and also
~\eqref{e:105.70} for $v=0$. To get~\eqref{e:105.70} for general $v$, we can simply write,
\begin{equation}
	\rmP^{\uparrow u}_{\infty, v} \circ \Pi^{-1}_{\bbT_r} = 
	\big(\rmP^{\uparrow u-v}_{\infty} (\cdot - v)\big) \circ \Pi^{-1}_{\bbT_r}\big) = 
	\big(\rmP^{\uparrow u-v}_{\infty} \circ \Pi^{-1}_{\bbT_r}\big) (\cdot - v)
	= \rmP_r^{\varphi_{\infty, r, u-v}}(\cdot - v) =  \rmP_{r,v}^{\varphi_{\infty, r,u}}(\cdot) \,.
\end{equation}
Lastly, for~\eqref{e:107.27}, it is enough to check that for all $k \geq 1$ and $r \geq |x|+k$,
\begin{equation}\label{e:enough}
\rmP_v^{\uparrow u}\big(h_{{\bbT_k(x)}} \in \cdot \,\big|\, h_{\bbT_r \setminus \bbT(x) \cup \{x\}} \big) = \rmP^{\uparrow u - c_0 |x|}_{h(x)} \circ \Pi_{\bbT_{k}}^{-1} 
\quad \text{a.s.}\,.
\end{equation}
Using~\eqref{e:100.160}, the left hand side above is equal to
\begin{equation}
\rmP_{r,v}^{\varphi_{\infty,r,u}}\big(h_{\bbT_k(x)}\big) \in \cdot \,\big|\, h_{\bbT_r \setminus \bbT(x) \cup \{x\}} \big) = 
\rmP_{r-|x|, h(x)}^{\varphi_{\infty, r-|x|,u-c_0|x|}} \circ \Pi_{\bbT_k}^{-1}
= \rmP_{\infty, h(x)}^{\uparrow u-c_0|x|} \circ \Pi_{\bbT_{r-|x|}}^{-1} \circ \Pi_{\bbT_k}^{-1}\,,
\end{equation}
with all equalities holding almost-surely. The last term is equal to the right hand side of~\eqref{e:enough}, almost-surely as well.
\end{proof} 

\subsection{Localization in first generations}
\label{s:3.2}
Crucial to deriving asymptotics in law under the conditioning, both locally and globally, are the following two localization results, which were shown in~\cite{Work1}. The first theorem, which provides upper and lower tail bounds on $h$ under the conditioning, is one of the main results in~\cite{Work1}. Henceforth we abbreviate
\begin{equation}
\label{e:101.62}
	\hat{h}(x) \equiv h(x) - \mu_n(x) \,,
\end{equation}
with $\mu_n$ as in~\eqref{e:1.10a}.
\begin{theorem}[Theorem~\ref{1@t:1.3} in~\cite{Work1}] 
\label{t:1.3}
There exists $C,c \in (0, \infty)$ such that 
for all $n \geq 1$, $x \in \bbT_n$ and $u > 0$,
\begin{equation}
\label{e:1.9a}
c \exp \Big(-C \frac{u^2}{	\ol{\sigma}_n(x, u)} \Big)
\leq 
	\rmP^+_n \big(\hat{h}(x) > u \big) \leq C \exp \Big(-c \frac{u^2}{	\ol{\sigma}_n(x, u)} \Big)
\end{equation}
and 
\begin{equation}
\label{e:1.9b}
c \exp \Big(-C \frac{u^2}{	\ul{\sigma}_n(x)} \Big) 
\leq 
	\rmP^+_n \big(\hat{h}(x) < -u \big) \leq C \exp \Big(-c \frac{u^2}{	\ul{\sigma}_n(x)} \Big) \,,
\end{equation}
where 
\begin{equation}
\label{e:1.11a}
	\ul{\sigma}_n(x) := (|x| - l_n)^++1
	\  , \quad
	\ol{\sigma}_n(x, u) := \big(\log_2 (u \wedge n) - (l_n - |x|)^+\big)^+ + (|x|-l_n)^+ + 1 \,.
\end{equation}
If $|x| > l_n$, then the lower bound in~\eqref{e:1.9b} holds only up to $C(m_{|x|-l_n}+1)$.
\end{theorem}

The above theorem shows, in particular, that under the positivity constraint the field $h$ strongly concentrates  around its mean in the first $l_n$ generations. 
 As a consequence the height of the field under $\rmP_n^+$ is essentially independent at vertices which are not too close in graph distance. A particular formulation of this is captured in the next proposition. This is the key ingredient in~\cite{Work1} for controlling the covariances and in particular in showing that the heights at generation $l_n$ under the conditioning are essentially i.i.d. In this current work, it plays an essential role in the proof of the existence of a limiting law under $\rmP_n^+$, as well as of the LLN type results for the global statistics. 
\begin{proposition}[Proposition~\ref{1@p:4.2} and Proposition~\ref{1@p:4.2a} from~\cite{Work1} with $u = m_n$]
	\label{p:106.3}
For each $n \geq 1$, $0 \leq k < k' < l_n - C$ and $v, v' \in \bbR$, with $x_n \in \bbL_n$ and $x_j := [x_n]_j$,
\begin{equation}
		\Big\| \rmP^+_n \big(\hat{h}(x_{k'}) \in \cdot \,\big|\, \hat{h}(x_k) = v')\big)  - 
					\rmP^+_n \big(\hat{h}(x_{k'}) \in \cdot \,\big|\, \hat{h}(x_k) = v)\big) \Big\|_{\rm TV}  \leq C\rme^{-(c(k'-k)-|v|\vee |v'|)^+} \,,
	\end{equation}
where $C,c  \in (0,\infty)$ are universal constants.
\end{proposition}	

For completeness, let us also recall the covariance estimates from~\cite{Work1}.
\begin{theorem}[Theorem~\ref{1@t:1.4} in~\cite{Work1}]
\label{t:1.4}
There exists $c > 0$ such that for all $n \geq 1$ and $x,y \in \bbT_n$ with $|x|,|y| \geq l_n$, 
\begin{equation}
\label{e:1.14}
\rmCov_n^+\, \big(h(x), h(y) \big) 
= \left\{ \begin{array}{lll}
|x \wedge y| - l_n + O(1)
& \quad & |x \wedge y| \geq l_n \,,\\
O\big( \rme^{-c(l_n - |x \wedge y|)}\big)
& \quad & |x \wedge y| < l_n \,.
\end{array} \right.
\end{equation}
\end{theorem}

\subsection{Proof of main results}
	Theorem~\ref{t:1.7} and Proposition~\ref{p:1.7e}, both follow easily from the next lemma.
\begin{lemma}
\label{l:101.7}
For all log-dyadic $\delta \in [0,1),$ there exists a law $\rmP^{+,\delta}$ on $\bbR^{\bbT}$ such that for all $r \geq 1$, 
\begin{equation}
\label{e:107.18}
\lim_{k \to \infty}	\Big|\rmE^{\uparrow c_0 k} F\Big(\hat{h}_{{\bbT_r(x_{k})}} 
\, \Big|\, h(0) = -c_0 \delta 2^{k} \Big)  - \rmE^{+,\delta} F\big(h_{{\bbT_r}}\big)\Big| =
\lim_{n \to \infty}
	\Big| \rmE_{n}^+ F\big(\hat{h}_{{\bbT_r(x_{l_n})}}\big) - \rmE^{+,[n]_2} F\big(h_{{\bbT_r}}\big)\Big| = 0\,,
\end{equation}
uniformly in all log-dyadic $\delta \in [0,1)$ and $F: \bbR^{\bbT_r} \to \bbR$ satisfying
\begin{equation}
\label{e:107.2a}
	F(h) \leq C \rme^{C\|h\|} \,,
\end{equation}
for any fixed $C < \infty$. Above $x_l$ denotes a generic vertex in $\bbL_l$.
Moreover, 
\begin{equation}
\label{e:107.31a}
	\rmP^{+,\delta}\big(\cdot) = \int \rmP^{\uparrow}_v\big(\cdot\big)\, \rmP^{+,\delta} \big(h(0) \in \rmd v\big) \,.
\end{equation}
\end{lemma}

\begin{proof}
We shall first show~\eqref{e:107.18} with~\eqref{e:107.2a} replaced by the condition
\begin{equation}
\label{e:107.7}
	\|F\|_\infty < C \,.
\end{equation}
To this end, fix $r \geq 0$, $k > 0$ and for each $n \geq 1$, set $k' := l_n - k$, $l_n' := l_n - C'$ for some $C'$ to be determined later. Recall that $x_{l_n} \in \bbL_{l_n}$ and abbreviate $x_{k'} \equiv [x_{l_n}]_{k'}$, $x_{l_n'} \equiv [x_{l_n}]_{l_n'}$. Given $F: \bbR^{\bbT_r} \to \bbR$ which satisfies~\eqref{e:107.7}, let
\begin{equation}
	\tilde{F}(u) := \rmE_n^+\Big(F\big(\hat{h}_{{\bbT_r(x_{l_n})}}\big) \,\big|\, h(x_{l_n'}) = u\Big) \,.
\end{equation}
Observe that $\tilde{F}$ also satisfies~\eqref{e:107.7}.

By conditioning on $h(x_{k'})$, the Tower Property, Proposition~\ref{p:106.3} and choosing $C'$ large enough, for all $n \geq 0$ we have

\begin{align}
\label{e:5.75}
\Big| \rmE_n^+ F\big(\hat{h}_{{\bbT_r(x_{l_n})}}\big) - \rmE_n^+ \Big(F\big(\hat{h}_{{\bbT_r(x_{l_n})}}\big)\,\Big|\, \hat{h}(x_{k'}) = 0\Big) \Big|
& = 
\Big| \rmE_n^+ \tilde{F}\big(\hat{h}(x_{l'_n})\big) - \rmE_n^+ \Big(\tilde{F}\big(\hat{h}(x_{l'_n})\big)\,\Big|\, \hat{h}(x_{k'}) = 0\Big) \Big|\nonumber \\
& \leq C \rmE_n^+ \rme^{-ck'+|\hat{h}(x_{k'})|} \leq C' \rme^{-c'k} \,.
\end{align}
Above, for the last inequality, we also used the almost-Gaussian tails of $\hat{h}(x_{k'})$ under $\rmP_n^+$, as guaranteed by Theorem~\ref{t:1.3}.
At the same time, using that 
\begin{equation}
\mu_n(x_{l_n}) - \mu_n(x_{k'}) = 2^{-k'} m_{n'} = 2^{-l_n+k} m_{n'}	
= c_0 2^{[n]_2} 2^{k} + O(2^{k} \log n/n) 
\end{equation}
and
\begin{equation}
\begin{split}
\Big\{\mu_n(x_{k'}) + \min_{\bbL_{n-k'}} h \geq 0\Big\}	
& = \Omega_{n-k'}\Big(m_{n-k'} - m_{n'} + m_{n'} - \mu_n(x_{k'}) \Big) \\
& = \Omega_{n-k'}\Big(c_0 \big(2^{[n]_2} 2^{k} + k\big)+ O(2^{k} \log n/n) \Big) \,,
\end{split}
\end{equation}
we may rewrite the second expectation on the left most side in~\eqref{e:5.75} as
\begin{multline}
\label{e:5.80}
	\rmE_{n-k'}\Big(F\big(h_{{\bbT_r(x_k)}} - c_0 2^{[n]_2} 2^{k} + O(2^{k} \log n/n)\big)
\,\Big|\, \Omega_{n-k'}\big(c_0 \big(2^{[n]_2} 2^{k} + k\big)+ O(2^{k} \log n/n) \big) \Big) \\
= \rmE^{\uparrow c_0 k} F\Big(h_{{\bbT_r(x_k)}} 
\, \Big|\, h(0) = -c_0 2^{[n]_2} 2^{k}
\Big) + o(1) \,.
\end{multline}
where $o(1) \to 0$ as $n \to \infty$ and $x_k \in \bbT_{k}$. The last equality follows from Lemma~\ref{l:5.7}.

Combining~\eqref{e:5.75} and~\eqref{e:5.80}, we get that
\begin{equation}
\label{e:5.81}
\lim_{k \to \infty} \limsup_{n \to \infty}
	\Big| \rmE_{n}^+ F\big(\hat{h}_{{\bbT_r(x_{l_n})}}\big) - \rmE^{\uparrow c_0 k} F\Big(h_{{\bbT_r(x_k)}} 
\, \Big|\, h(0) = -c_0 2^{[n]_2} 2^{k}
\Big)
 \Big| = 0 \,,
\end{equation}
uniformly in all $F$ satisfying~\eqref{e:107.7}.

Now, for any log-dyadic $\delta \in [0,1),$ we may find an increasing sequence $(n_j)_{j \geq 1}$ of integers such that $[n_j]_2 = \delta$ for all $j \geq 1$. Using this in~\eqref{e:5.81} shows that the sequence
\begin{equation}
\label{e:107.20}
\rmE^{\uparrow c_0 k} F\Big(h_{{\bbT_r(x_k)}} 
\, \Big|\, h(0) = -c_0 2^{[n]_2} 2^{k} \Big)  
\quad ; \qquad k \geq 0
\end{equation}
is Cauchy, uniformly in all dyadic $\delta$ and $F$ as above, and therefore tends as $k'^\prime \to \infty$, uniformly to a limit. 

Restricted to bounded continuous functions $F$, the $k \to \infty$ limit of~\eqref{e:107.20} is a bounded linear functional, mapping positive functions to positive functions and the function $F\equiv 1$ to $1$. As such, there exists a probability measure $\rmP^{+,\delta, r}$ on $\bbR^{\bbT_r}$ such that this limit is equal to $\rmE^{+,\delta,r} F(h)$ for all such $F$. In view of~\eqref{e:5.81}, we thus get
\begin{equation}
\label{e:107.15}
\lim_{k \to \infty}
	\bigg| \rmE^{\uparrow c_0 k} F\Big(h_{{\bbT_r(x_k)}} 
\, \Big|\, h(0) = -c_0 \delta 2^{k} \Big) - \rmE^{+,\delta,r} F(h)\bigg|  =
\lim_{n \to \infty}
	\bigg| \rmE_{n}^+ F\big(\hat{h}_{{\bbT_r(x_{l_n})}}\big) - \rmE^{+,[n]_2, r} F(h)\bigg| = 0
\end{equation}
uniformly in all continuous $F$ satisfying~\eqref{e:107.7} and log-dyadic $\delta \in [0,1)$.
Taking $\bbL^1$ approximations by continuous functions simultaneously with respect to all measures in~\eqref{e:107.15} and using the uniformity in $F$, the above limit also holds uniformly for all $F$ satisfying~\eqref{e:107.7} which are not necessarily continuous.

By consistency of the pre-limits, it holds that the family of measures $\{\rmP^{+,\delta,r} :\: \r \geq 0\}$ are consistent and as such, by Kolmogorov's Extension Theorem, for all $\delta$ as above, we may define a probability measure $\rmP^{+,\delta}$ on $\bbR^{\bbT}$ having $\rmP^{+,\delta,r}$ as its projection onto $\bbR^{\bbT_r}$. We thus recover~\eqref{e:107.18} albeit under~\eqref{e:107.7}.

To get uniformity w.r.t. to all functions satisfying~\eqref{e:107.2a}, we first observe that the convergence in~\eqref{e:107.18} implies that $\rmP^{+,\delta}$ and 
$\rmP^{\uparrow c_0 k}_{-c_0 \delta 2^{k}} \big(h_{{\bbT(x_k)}} \in \cdot\big)$ inherit the uniform tail bounds of the marginals of $\rmP_n^+\big(h_{{\bbT(x_{l_n})}} \in \cdot\big)$, given by Theorem~\ref{t:1.3}. This implies that 
the random variable $F(h)$ under the above three laws forms a uniformly integrable family, and in particular, we may replace any $F:\bbR^{\bbT_r} \to \bbR$ satisfying~\eqref{e:107.2a} by $F(h)\,\1_{\{\|h\| \leq M\}}$ with all means in~\eqref{e:107.15} changing by an arbitrarily small constant, as long as $M$ is  large enough. As the latter function is now bounded, appealing to~\eqref{e:107.18} for bounded functions, we thus obtain~\eqref{e:107.18} in the desired generality, by taking first $n \to \infty$ followed by $M \to \infty$.

Lastly, to get~\eqref{e:107.31a}, we fix $r$ and $F: \bbR^{\bbT_r} \to \bbR$ bounded, and for $l \in [0,\infty]$, $M > 0$, define $\tilde{F}_{l,M}: \bbR \to \bbR$ via
\begin{equation}
	\tilde{F}_{l,M}(u) := \big(\rmE_{l,u}^{\uparrow} F(h_{{\bbT_r}})\big)\1_{\{|u| \leq M\}}\,.
\end{equation}
Now, by the Tower Property and the Triangle Inequality,
\begin{equation}
\begin{split}
\Big| \rmE^{+,[n]_2} & F\big(h_{{\bbT_r}}\big) - \rmE^{+,[n]_2}\big(\rmE^{\uparrow}_{h(0)} F(h_{{\bbT_r}})\big) \Big| \leq 
\Big| \rmE^{+,[n]_2} F\big(h_{{\bbT_r}}\big) - \rmE_n^+ F\big(\hat{h}_{{\bbT_r(x_{l_n})}}\big) \Big|\\ 
& + \Big| \rmE_n^+ F\big(\hat{h}_{{\bbT_r(x_{l_n})}}\big) 
- \rmE_n^+ \tilde{F}_{n', M}\big (\hat{h}(x_{l_n}) \big) \Big| 
 + 
\Big| \rmE_n^+ \tilde{F}_{n',M}\big (\hat{h}(x_{l_n}) \big)
- \rmE_n^+ \tilde{F}_{\infty,M}\big (\hat{h}(x_{l_n}) \big) 
\Big| \\
& + \Big|\rmE_n^+ \tilde{F}_{\infty,M}\big (\hat{h}(x_{l_n}) \big) -
	 \rmE^{+,[n]_2} \tilde{F}_{\infty,M}\big(h(0)\big) \Big|  
+ \Big|\rmE^{+,[n]_2} \tilde{F}_{\infty,M}\big(h(0)\big) - \rmE^{+,[n]_2}\big(\rmE^{\uparrow}_{h(0)} F(h_{{\bbT_r}})\big) \Big|\,.
\end{split}
\end{equation}
The first term on the right hand side tends to $0$ as $n \to \infty$ by~\eqref{e:107.18}. The second tends to $0$ as $M \to \infty$, uniformly in $n$, thanks to the almost Gaussian tails of $\rmP_n^+\big(\hat{h}(x_{l_n}) \in \cdot)$ which are a consequence of Theorem~\ref{t:1.3}. These tails carry over to the limiting law $\rmP^{+,\delta}(h(0) \in \cdot)$ thanks to~\eqref{e:107.18} with $r=0$, and therefore also the last term tends to $0$ as $M \to \infty$. The third term tends to $0$ as $n \to \infty$ for fixed $M$ thanks to the uniform convergence in Lemma~\ref{l:5.7}, while the forth term tends to $0$ in the same sense thanks to~\eqref{e:107.18} with $k'=0$. Altogether, this gives~\eqref{e:107.18} for general $k' \geq 1$ and uniformly in $F$, satisfying~\eqref{e:107.7}. 
\end{proof}

We are now ready for
\begin{proof}[Proof of Theorem~\ref{t:1.7}, Proposition~\ref{p:1.7e} and Proposition~\ref{p:1.11}]
Arguing in the same way as in the proof of Proposition~\ref{p:1.45}, the limits in~\eqref{e:107.18} which hold for all polynomially growing functions $F$, readily imply Proposition~\ref{p:1.7e} and~\eqref{e:1.17} of Theorem~\ref{t:1.7}. Relation~\eqref{e:107.31} is precisely~\eqref{e:107.31a}. Relation~\eqref{e:107.30} is then a consequence of~\eqref{e:107.31} and~\eqref{e:107.27} of Proposition~\ref{p:1.45}. Taking $F = \1_{\{h(x) > u\}}$ and
	$F = \1_{\{h(x) < -u\}}$, for $x \in \bbT$ and $u \in \bbR$ in the second limit of~\eqref{e:107.18} together with Theorem~\ref{t:1.3} shows Proposition~\ref{p:1.11}.
\end{proof}

Turning to the proofs of Theorem~\ref{t:7.7}, Theorem~\ref{t:107.4} and Theorem~\ref{p:107.5}, the following lemma is key.
\begin{lemma}
\label{l:107.3}
For all $p \geq 1$, there exist $C_p \in (0,\infty)$ such that for all
$n \in [1, \infty]$ and $u,v,w \in \bbR$ and with $y \in \bbL_1$,
\begin{equation}
\label{e:7.24}
	W^{p,\infty}\Big(\rmP_{n,w}\big(h(y) - w \in \cdot) ,\, \rmP^{\uparrow u}_{n,v}\big(h(y) - v \in \cdot)\Big) \leq 
	C_p \big(((u-v)^-)^{1/p} + 1\big) \rme^{c_0 (u-v)/p}.
\end{equation}	
In particular, there exists a coupling $\rmP_{n,v,w}^{\uparrow u}$ of random fields $h$, $h^{\uparrow u}_{n,v}$ and $\eta_{n,v}^{\uparrow u}$ on $\bbT_n$ such that:
\begin{enumerate}
	\item $\rmP_{n,v,w}^{\uparrow u}\big(h(y) \in \cdot\big) = \rmP_{n,w}\big(h(y) \in \cdot)$.\item $\rmP_{n,v,w}^{\uparrow u}\big(h^{\uparrow u}_{n,v}(y) \in \cdot\big) = \rmP^{\uparrow u}_{n,v}\big(h(y) \in \cdot\big)$.
\item
	$\eta_{n,v}^{\uparrow u}(y) = h^{\uparrow u}_{n,v}(y) - h(y)$.
\item 
	$\big\|\eta_{n,v}^{\uparrow u}(y) - (v-w)\big\|_p \leq C_p \big(((u-v)^-)^{1/p} + 1\big) \rme^{c_0 (u-v)/p}$.
\item $\rmE_{n,v,w}^{\uparrow u} \big(\eta_{n,v}^{\uparrow u}(y)\big) \geq v-w$.
\end{enumerate}
\end{lemma}

\begin{proof}
Let $n \in [1,\infty)$. Without loss of generality, we can assume that $w=v=0$. Indeed, the law of 
$\rmP_{n,w}\big(h(y) \in \cdot)$ is invariant in $w$ and $\rmP^{\uparrow u}_{n,v}\big(h(y) - v \in \cdot)
= \rmP^{\uparrow u-v}_{n,0}\big(h(y) - (u-v) \in \cdot)$. Now, while the law of $h(x)$ under $\rmP_n \equiv \rmP_{n,0}$ is that of a standard Gaussian, the law of $h(x)$ under $\rmP_n^{\uparrow u} = \rmP_{n,0}^{\uparrow}$ is absolutely continuous with respect to the latter with Radon-Nikodym derivative given by
\begin{equation}
	\frac{p_{n-1}(u - m_n + m_{n-1} - \cdot)}{p_n(u)}\,.
\end{equation}
Letting $f_{\cN(0,1)}$ be the standard Gaussian density, by, e.g., Theorem 6.15 in~\cite{villani}, the left hand side in~\eqref{e:7.24} is therefore at most a $p$ dependent constant times
\begin{equation}
\label{e:107.26}
	\int |s|^p f_{\cN(0,1)}(s) 	\bigg| 1 - \frac{p_{n-1}(u - m_n + m_{n-1} - s)}{p_n(u)}\bigg| \rmd s\,.
\end{equation}
Now, if $u < 0$, we upper bound the absolute value above by 
\begin{equation}
\begin{split}
	\bigg|\frac{q_{n-1}(s + m_n - m_{n-1} - u) - q_n(-u)}{p_n(u)} \bigg|
	& \leq C \big((-u+s)^+ \rme^{-c_0 (-u+s)^+} + (-u) \rme^{-c_0(-u)}\big) \\
	& \leq C' \big(-u + s^+\big) \rme^{-c_0(-u)} \,.
\end{split}
\end{equation} 
where $q$ is as in~\eqref{e:103.24} and where we used the tightness of the centered minimum.
Plugging this in~\eqref{e:107.26} gives $C (u^-+1) \rme^{-c_0 u^-}$ as an upper bound.

At the same time, if $u \geq 0$, we bound the absolute value by the sum of the terms inside it. Performing the integration, this gives a constant plus 
\begin{equation}
\begin{split}
\rmE_n^{\uparrow u}(|h(x)|^p) & \leq
C' \big( \rmE_n^{\uparrow u}(h(x)^+)^p + \rmE_n^{\uparrow u}(h(x)^-)^p \big) 
\leq C'' \big(u^p + \rmE_n \big(h(x)^+)^p + \rmE_n (h(x)^-)^p \big) \\
& \leq C''' (u^p + 1) \,,
\end{split}
\end{equation}
where we used Lemma~\ref{l:2.4}, FKG  and the fact that $h(x)$ under $\rmP_n$ is a standard Gaussian. Summing the two bounds, taking the $p$-th root and using that $(a+b)^{1/p} \leq a^{1/p} + b^{1/p}$ for $p \geq 1$ and $a,b \geq 0$, gives~\eqref{e:7.24}. Existence of the desired coupling follows by definition of the Wasserstein distance. Part (5) follows by the FKG property of $\rmP_n$.

The case $n=\infty$ follows exactly in the same way. Replacing $p_n$, $p_{n-1}$ with $p_\infty$ and $q_n$, $q_{n-1}$ with $q_\infty$ noting that $p_\infty$ and $q_\infty$ obey the same bounds being the limit of the corresponding finite $n$ quantities. Lemma~\ref{l:2.4} extends to the case $n=\infty$ by a standard limiting procedure.
\end{proof}

The following is an analog of Theorem~\ref{t:107.4} for the laws $\rmP_{n,v}^{\uparrow u}$. 
Theorems~\ref{t:7.7}, Theorem~\ref{t:107.4} and Theorem~\ref{p:107.5} will be a direct consequence thereof.
\begin{proposition}
\label{p:7.9}
For all $n \in [1,\infty]$ and $u,v \in \bbR$, there exists a coupling $\ddot{\rmP}^{\uparrow u}_{n,v}$ between the fields $h$, $h^{\uparrow u}_{n,v}$ and $\eta^{\uparrow u}_{n,v}$ on $\bbT_n$ such that 
\begin{enumerate}
	\item $\ddot{\rmP}^{\uparrow u}_{n,v}$-almost-surely,
	\begin{equation}
	\label{e:107.46}
		h^{\uparrow u}_{n,v} = h + \eta^{\uparrow u}_{n,v} \,.
	\end{equation} 
	\item The laws of $h$, $h^{\uparrow u}_{n,v}$ and $\eta^{\uparrow u}_{n,v}$ obey
	\begin{equation}
	\label{e:107.45}
		 \ddot{\rmP}^{\uparrow u}_{n,v} \big(h \in \cdot) = \rmP_{n,0}(h \in \cdot)
		 \,, \qquad 
		\ddot{\rmP}^{\uparrow u}_{n,v} \big(h^{\uparrow u}_{n,v}\in \cdot) = \rmP^{\uparrow u}_{n,v}(h \in \cdot) \,,
	\end{equation}
	and
	\begin{equation}
		 \ddot{\rmP}^{\uparrow u}_{n,v} \big(\eta_{n,v}^{\uparrow u}(0) \in \cdot\big) = \rmP^{\uparrow u}_{n,v}\big(h(0) \in \cdot\big) \,.
	\end{equation}
	\item For all $x, y \in \bbT$ with $x$ a direct child of $y$ and $p \geq 1$,
	\begin{equation}
		\label{e:107.47}
		\big\|\eta^{\uparrow u}_{n,v}(x) - \eta^{\uparrow u}_{n,v}(y) \big\|_p \leq 
		C_p \big(((u-v)^-)^{1/p} + 1\big) \rme^{c_0 (u-v)/p}\,|x|\, 2^{-|x|/p}
	\end{equation}
	and
	\begin{equation}
			\label{e:107.47a}
			\ddot{\rmE}_{n,v}^{\uparrow u} \Big(\eta^{\uparrow u}_{n,v}(x) - \ddot{\rmE}_{n,v}^{\uparrow u} \,\eta^{\uparrow u}_{n,v}(y)\Big) \geq 0 \,,
		\end{equation}
	where $C_p < \infty$ is independent of $\delta$.
	\end{enumerate}
	Moreover, the following relations between the above couplings hold:
	\begin{enumerate}
		\item[(4)] For all $u,v \in \bbR$, $n < \infty$ and $x \in \bbT_n$, almost-surely,
	\begin{equation}
	\label{e:107.48}
		\ddot{\rmP}_{n,v}^{\uparrow u} \Big(\Big(h - h(x),\, \eta^{\uparrow u}_{n,v} + h(x),\, h^{\uparrow u}_{n,v}\Big)_{\bbT(x) \cap \bbT_n} \in \cdot 
		\,\Big|\, \big(h,\, h_{n,v}^{\uparrow u},\, \eta_{n,v}^{\uparrow u}\big)_{{\{x\} \cup (\bbT_n \setminus \bbT(x))}}\Big) \nonumber\\
		= 
		\ddot{\rmP}_{n-|x|,\, h^{\uparrow u}_{n,v}(x)}^{\uparrow (u - m_n + m_{n-|x|})} \,.
	\end{equation}
	The same holds when $n = \infty$ with the right hand side in~\eqref{e:107.48} replaced by
\begin{equation}
		\ddot{\rmP}_{\infty,\, h^{\uparrow u}_{\infty, v}(x)}^{\uparrow u - c_0|x|}
		\equiv \ddot{\rmP}_{h^{\uparrow u}_{v}(x)}^{\uparrow u - c_0|x|} \,.
\end{equation}
		\item[(5)] For all $u,v \in \bbR$ and $n \in [1,\infty]$,
		\begin{equation}
			\label{e:107.49}
			\ddot{\rmP}_{n,v}^{\uparrow u}
			\Big(\big(h,\, \eta^{\uparrow u}_{n,v},\, h^{\uparrow u}_{n,v} \big) \in \cdot \Big) =
			\ddot{\rmP}_{n,0}^{\uparrow u-v}
			\Big(\big(h,\, \eta^{\uparrow u-v}_{n,0} + v,\,h^{\uparrow u-v}_{n,0} + v \big) \in \cdot \Big) \,.
	\end{equation}
	\end{enumerate}
\end{proposition}

\begin{proof}
Fix $n < \infty$, $u,v \in \bbR$. We draw the three fields $(h, h^{\uparrow u}_{n,v},\eta^{\uparrow u}_{n,v})$ iteratively in a Markovian manner. We start by setting $h(0) = 0$ and $h^{\uparrow u}_{n,v}(0) = \eta^{\uparrow u}_{n,v}(0) = v$. Then for any $y \in \bbT_n \setminus \{0\}$, whenever the values of the fields have been set for $y$'s direct ancestor $x$, we draw the values of $\big(h(y), h^{\uparrow u}_{n,v}(y), \eta^{\uparrow u}_{n,v}(y)\big)$ according to their law under $\rmP^{\uparrow u - m_n + m_{n-|x|}}_{n-|x|,  h_{n,v}^{\uparrow u}(x), h(x)}$ as given by Lemma~\ref{l:107.3}, independently of all other drawings. The order of drawings is immaterial. 
	
Now Parts (1), (2) and (4) follow from the construction, as by the Markov property of $\rmP_n$, $\rmP_n^+$ we have
\begin{equation}
	\rmP_n \big(h(y) \in \cdot \,\big|\, h_{{\{x\}\cup(\bbT_n \setminus \bbT(x)})} \big) = 
		\rmP_{n-|x|, h(x)} \big(h(z) \in \cdot)
\end{equation}
and
\begin{equation}
	\rmP_n^{\uparrow u} \big(h(y) \in \cdot \,\big|\, h_{{\{x\}\cup(\bbT_n \setminus \bbT(x)})} \big) = 
		\rmP_{n-|x|, h(x)}^{\uparrow u- m_n + m_{n-|x|}} \big(h(z) \in \cdot) \,,
\end{equation}
where $x$ and $y$ are as in the construction, $z \in \bbL_1$ and both equalities hold almost-surely. The second statement in part (3) follows, since by construction and Part (5) of Lemma~\ref{l:107.3},
\begin{equation}
\ddot{\rmE}_{n,v}^{\uparrow u} \eta^{\uparrow u}_{n,v}(y) = 	\ddot{\rmE}_{n,v}^{\uparrow u}\Big(	\ddot{\rmE}_{n,v}^{\uparrow u}\Big(\eta^{\uparrow u}_{n,v}(y)\,\Big|\, h^{\uparrow u}_{n,v}(x),  h(x), \eta^{\uparrow u}_{n,v}(x) \Big)\Big)
	 \geq \ddot{\rmE}_{n,v}^{\uparrow u}\Big(h^{\uparrow u}_{n,v}(x) - h(x)\Big) = 	\ddot{\rmE}_{n,v}^{\uparrow u} \eta^{\uparrow u}_{n,v}(x)  \,.
\end{equation}
Part (5) follows by construction and~\eqref{e:101.20}.

As for the first inequality in Part (3), by construction and Lemma~\ref{l:107.3}, with the right hand side in~\eqref{e:7.24} abbreviated as $f_{n,p}(u,v)$, we have
\begin{equation}
\begin{split}
	\ddot{\rmE}^{\uparrow u}_{n,v} \big|\eta^{\uparrow u}_{n,v}(x) - \eta^{\uparrow u}_{n,v}(y)\big|^p 
	& = \ddot{\rmE}^{\uparrow u}_{n,v} \Big(\ddot{\rmE}^{\uparrow u}_{n,v} \Big(\big|\eta^{\uparrow u}_{n,v}(y) - \eta^{\uparrow u}_{n,v}(x)\big|^p 
	\,\Big|\, h(x), h^{\uparrow u}_{n,v}(x), \eta^{\uparrow u}_{n,v}(x)\Big)\Big) \\
	& \leq \rmE_{n, v}^{\uparrow u} f_{n-|x|,p} \big(u - m_n + m_{n-|x|}, h(x) \big) \\
	& \leq \rmE_{n, 0}^{\uparrow u-v} f_{n-|x|,p} \big(u-v - m_n + m_{n-|x|}, h(x) \big) \\
	& \leq C \rme^{c_0 (u-v)} \rme^{-c_0^2 |x|} \big(\big(c_0 |x| - (u-v)\big)^+ + 1 \big)
		\rmE_n^{\uparrow u-v} \big(\big(h(x)^+ + 1 \big) \rme^{-c_0 h(x)}\big) \\
	& \leq C ((u-v)^- + 1) \rme^{c_0 (u-v)}\,|x| \rme^{-c_0^2 |x|} \,
		\rmE_n \rme^{-c_0 h(y)} \\
	& \leq C ((u-v)^- + 1) \rme^{c_0 (u-v)}\,|x| \rme^{-(c_0^2/2) |x|}  \,,
\end{split}
\end{equation}
where we used the FKG property of $\rmP_n$ to get rid of the conditioning on $\Omega_n(u)$ and the fact that under $\rmP_n$ the law of $h(y)$ is $\cN(0,|y|)$. Taking the $p$-th root, this gives~\eqref{e:107.47}.

The case $n=\infty$ is treated in the same way. We draw $\big(h(y), h^{\uparrow u}_{\infty,v}(y), \eta^{\uparrow u}_{\infty,v}(y)\big)$ conditional on $\big(h(x), h^{\uparrow u}_{\infty,v}(x), \eta^{\uparrow u}_{\infty,v}(x)\big)$ according to the coupling law $\rmP^{\uparrow u - c_0|x|}_{\infty,  h_{\infty, v}^{\uparrow u}(x), h(x)}$ from Lemma~\ref{l:107.3} with $n=\infty$. 
\end{proof}
We can now give
\begin{proof}[Proof of Theorem~\ref{t:7.7}]
Fix $p \geq 1$, and for each $1 \leq r \leq n$ and any $x \in \bbL_{l_n}$, we construct a coupling $\ddot{\rmP}^+_{n,\infty}$ between the fields $h^+_n$ and $h^+_\infty$ on $\bbT_{n'}$ as well as the fields $h$, $\eta^+_n$ and $\eta^+_\infty$ on $\bbT_{n'} \setminus \bbT_r$.  
To this end, we first draw $\big(h^+_n, h^+_\infty\big)_{\bbT_r}$ according to a joint law which realizes the $W^{p,\infty}$ distance between $\rmP_n^+\big(h_{{\bbT_r(x)}} - m_{n'} \in \cdot \big)$ and $\rmP^{+,[n]_2}\big(h_{{\bbT_r}} \in \cdot\big)$. Then, independently for each $y \in \bbL_r$ we draw jointly $\big(h,\, h^+_n,\, \eta^+_n\big)_{\bbT_{n'-r}(y)}$
and $\big(h, h^+_\infty,\, \eta^+_\infty\big)_{\bbT_{n'-r}(y)}$ according to the laws $\ddot{\rmP}^{\uparrow - m_{n'} + m_{n'-r}}_{n'-r, h^+_n(y)}$ and $\ddot{\rmP}^{\uparrow - c_0 r}_{\infty, h^{+,\delta}(y)}$ respectively. Notice that the first component of both triplets is the same field. This is always possible as both coupling laws have the same marginal for the first component, which is that of a standard BRW.

We first claim that under this coupling the laws of $h^+_n$ and $h^+_\infty$ are $\rmP_n^+\big(h_{{\bbT_{n'}(x)}} - m_{n'} \in \cdot \big)$ and $\rmP^{+,[n]_2}\big(h_{{\bbT_{n'}}} \in \cdot\big)$ respectively. This will follow from the construction once we verify that
for all $y \in \bbL_r$ almost-surely,
\begin{equation}
	\rmP_n^+ \big(h_{{\bbT_{n'-r}(y)}} - m_{n'} \in \cdot \,\big|\, 
	h_{{\bbT_n \setminus \bbT_{n'-r}(y)}} \big) = 
	\rmP_{n'-r, h(y)}^{\uparrow m_{n'-r}} \big(h - m_{n'} \in \cdot \big) = 
	\rmP_{n'-r, h(y)-m_{n'}}^{\uparrow -m_{n'} + m_{n'-r}}
\end{equation}
and
\begin{equation}
	\rmP^{+,[n]_2} \big(h_{{\bbT(y)}} \in \cdot \,\big|\, 
	h_{\{y\} \cup (\bbT \setminus \bbT(y))} \big) = 
	\rmP_{h(y)}^{\uparrow -c_0 r} \,.
\end{equation}
Indeed, the first display holds thanks to the spatial Markov property of the standard BRW and the definition of $\rmP_n^+$, and the second thanks to~\eqref{e:107.30} of Theorem~\ref{t:1.7}.

At the same time, also by construction,
\begin{equation}
\label{e:107.51}
\begin{split}
	\sup_{z \in A_n}& \big|h_n^+(z) - h_\infty^+(z)\big| \\
	 & \leq 2\sup_{z \in A_n \cap \bbT_r} \big|h_n^+(z) - h_\infty^+(z)\big| 
		+ \sup_{z \in A_n \setminus \bbT_r} 
		\big|\big(h_n^+(z) - h^+_n([z]_r)\big) - \big(h_\infty^+(z) - h_\infty^+([z]_r)\big)\big| \\
		& \leq 2\sup_{z \in A_n \cap \bbT_r} \big|h_n^+(z) - h_\infty^+(z)\big| + 
		\sup_{z \in A_n \setminus \bbT_r} \Big(\big|\eta_n^+(z) - \eta_n^+([z]_r)\big| + 
			\big|\eta_\infty^+(z) - \eta_\infty^+([z]_r) \big|\Big) \\
		& \leq 2\sup_{z \in A_n \cap \bbT_r} \big|h_n^+(z) - h_\infty^+(z)\big| + 
		\sum_{k=r+1}^{n'} \sum_{\substack{z \in A_n \cap \bbL_k\\ y = [z]_{k-1}}}
			 \big|\eta_{n}^+(y) - \eta_{n}^+(z)\big| + \big|\eta_\infty^+(y) - \eta_\infty^+(z) \big| \,.
\end{split}
\end{equation}

It follows from the Markov property of the coupling in Proposition~\ref{p:7.9}, namely~\eqref{e:107.48}, that
\begin{equation}
	\ddot{\rmP}^+_{n,\infty}\big(\eta_n^+ \in \cdot\big) = \int \ddot{\rmP}_{n', v}^{\uparrow} \Big(\big(\eta_{n',v}^{\uparrow}\big)_{\bbT_{n'} \setminus \bbT_r} \in \cdot\Big) \rmP_n^+ \big(h(x) - m_{n'} \in \rmd v\big) \,,
\end{equation}
and
\begin{equation}
	\ddot{\rmP}^+_{n,\infty}\big(\eta_\infty^+ \in \cdot\big)	 = 
	\int \ddot{\rmP}^{\uparrow}_v \Big(\big(\eta_{v}^{\uparrow}\big)_{\bbT_{n'} \setminus \bbT_r} \in \cdot\Big) \rmP^{+,[n]_2} \big(h(x) \in \rmd v\big) = 
	\ddot{\rmP}^{+,[n]_2} 
	\Big(\big(\eta_\infty^+\big)_{\bbT_{n'} \setminus \bbT_r}  \in \cdot\Big) \,.
\end{equation}
It then follows from the construction and Proposition~\ref{p:7.9} that the $\bbL^p$ norm of the last right hand side in~\eqref{e:107.51} is at most 
\begin{multline}
2 W^{p,\infty} \Big(\rmP^+_n \big(h_{{\bbT_r(x)}} - m_{n'} \in \cdot\big),\,
				\rmP^{+,[n]_2} \big(h_{{\bbT_r}} \in \cdot\big)\Big) 
				+ 
				2C_p \sum_{k=r+1}^{n'} |A_n \cap \bbL_k| 2^{-k/p} k \\
				 \times \int \Big(\big((v^+)^{1/p} + 1\big) \rme^{-c_0 v/p}
				\Big(\rmP_n^+ \big(h(x) - m_{n'} \in \rmd v\big) +
				\rmP^{+,[n]_2} \big(h(x) \in \rmd v\big)\Big)\,.
\end{multline}				
The integral is bounded from above by a constant thanks to the almost-Gaussian tails (with constant coefficients) of both terms in the integrating measure, which is a consequence of Theorem~\ref{t:1.3}. Then thanks to~\eqref{e:107.58} and the assumption on $(A_n)_{n \geq 1}$, the last display tends to $0$ when $n \to \infty$ followed by $r \to \infty$.
\end{proof}

\begin{proof}[Proof of Corollary~\ref{c:1.9}]
By~\eqref{e:103.32a} the sum in~\eqref{e:107.40} is arbitrarily small for all $n \geq n_0$ uniformly in $r$, as long as $n_0$ is large enough. At the same time, the same sum for all $n \leq n_0$ is equal to zero whenever $r > n'_0$ as $A_n \subset \bbT_{n_0'}$. This implies~\eqref{e:107.40}. For the second assertion, 
since the series~\eqref{e:103.32} converges, its remainder tends to $0$, which implies~\eqref{e:107.40} with $A_n = A \cap \bbT_{n'}$.
\end{proof}

\begin{proof}[Proof of Corollary~\ref{c:1.11a} and Corollary~\ref{c:1.15a}]
With $A_n = A = \rmB_{r_{n',p}}(x_{n'}) \cap \bbT_{n'}$, the sum in~\eqref{e:103.32a} is equal to
\begin{equation}
\sum_{k = n'-{r_{n',p}}}^{n'} 2^{k-n'+{r_{n',p}}} 2^{-k/p} k 
= 2^{-n'+{r_{n',p}}} \sum_{k = n'-{r_{n',p}}}^{n'} 2^{k(1-1/p)} k 
\leq C 2^{-n'/p+{r_{n',p}}}n'
\leq C' n'^{-\epsilon},
\end{equation}
which tends to $0$ as $n \to \infty$.  
\end{proof}

\begin{proof}[Proof of Theorem~\ref{t:107.4}]
For $\delta \in [0,1)$, we define the desired coupling via
\begin{equation}
	\ddot{\rmP}^{+, \delta}\Big(\big(h, h^{+,\delta}, \eta^{+,\delta}\big) \in \cdot \Big)
	:= \int \ddot{\rmP}^{\uparrow}_v \Big(\big(h, h^{\uparrow}_v, \eta^{\uparrow}_v  \big) \in \cdot \Big)\rmP^{+,\delta}(h(0) \in  \rmd v) \,,
\end{equation}
where $\ddot{\rmP}^{\uparrow}_v \equiv \ddot{\rmP}^{\uparrow 0}_{\infty, v}$ and 
$\rmP^{+,\delta}$ are as given in Proposition~\ref{p:7.9} and Theorem~\ref{t:1.7} respectively.
Then (1), (2) follow from (1), (2) of Proposition~\ref{p:7.9} and~\eqref{e:107.31} of Theorem~\ref{t:1.7}. For (3), we compute
\begin{equation} 
\begin{split}
	\ddot{\rmE}^{+,\delta} \big|\eta^{+,\delta}(x) - \eta^{+, \delta}(y) \big|^p
		& = \int \ddot{\rmE}_v^\uparrow \big|\eta^{\uparrow}_v(x) - \eta^{\uparrow}_v(y) \big|^p \rmP^{+,\delta}(h(0) \in  \rmd v) \\
	& \leq C_p |x|^p\, 2^{-|x|} \int \big(v^+ + 1\big) \rme^{-c_0 v} \, \rmP^{+,\delta}(h(0) \in  \rmd v) 
	\leq C'_p |x|^p\, 2^{-|x|} \,.
\end{split}
\end{equation}
With the integral bounded by a constant, thanks to the almost Gaussian tails of $h(0)$ under $\rmP^{+,\delta}$ as asserted by Theorem~\ref{t:1.3}. The second inequality in~\eqref{e:101.39} is inherited from~\eqref{e:107.47a}.
\end{proof}

\begin{proof}[Proof of Theorem~\ref{p:107.5}]
It follows from (3) of Theorem~\ref{t:107.4} that the series
\begin{equation}
	\sum_{n \geq 1} \eta^{+,\delta}(x_n) - \eta^{+,\delta}(x_{n-1}) 
\end{equation}
converges absolutely in $\bbL^p$ for all $p \geq 1$ and therefore also conditionally and $\ddot{\rmP}^{+, \delta}$-almost-surely. Since Proposition~\ref{p:1.11} implies that $\eta^{+,\delta}(0)$ is in $\bbL^p$ for all $p \geq 1$, this shows~\eqref{e:107.37}. The same part of Theorem~\ref{t:107.4} also gives the monotonicity in $n$ of the mean of $\eta^{+,\delta}(x_n)$ and therefore also of $h^{+,\delta}(x_n)$, as $h$ is centered. The facts that neither the law of $\eta^{+,\delta}({\bf x})$ nor the means of $\eta^{+,\delta}(x_n)$ and $h^{+,\delta}(x_n)$ depend on the particular choice of ${\bf x}$, follow from the invariance of the law of $\eta^{+,\delta}$ under the tree symmetries, which is clear from the construction of the coupling. 

Now, if $A \subset \bbT$ is connected, then it must be a subtree of $\bbT$ rooted at some vertex $x_k \in \bbL_k$ for some $k \geq 1$. Taking any branch ${\bf x}$ passing through this $x_k$, we may write
\begin{equation}
\begin{split}
\sup_{y \in A} \big| \eta^{+,\delta}(y) - \eta^{+,\delta}({\bf x}) \big| 
& \leq \big|\eta^{+,\delta}(x_k) - \eta^{+,\delta}({\bf x}) \big| +
\sup_{y \in A} \big| \eta^{+,\delta}(y) - \eta^{+,\delta}(x_k) \big| \\
& \leq 
2 \sum_{n \geq {k+1}}
 \sum_{\substack{y \in A \cup \{x_n\} \cap \bbL_n\\ x = [y]_{n-1}}}
\big| \eta^{+,\delta}(y) - \eta^{+,\delta}(x)\big|\,.
\end{split}
\end{equation}
Using (3) of Theorem~\ref{t:107.4} and the Triangle Inequality, for any $p \geq 1$, the $\bbL^p$ norm of the right hand side above is at most
\begin{equation}
2C_p \sum_{n \geq {k+1}} \big(\big|A \cap \bbL_n|+1 \big) 2^{-n/p} n \,,
\end{equation}
which gives~\eqref{e:107.38} due to the law invariance of $\eta^{+,\delta}({\bf x})$ with respect to the choice of ${\bf x}$ and then also~\eqref{e:107.39} thanks to (2) of Theorem~\ref{t:107.4}.
\end{proof}

\begin{proof}[Proof of Corollary~\ref{c:1.15}]
Follows immediately from Corollary~\ref{c:1.11a} and Corollary~\ref{c:1.15a} and the Triangle Inequality applied to the distance $W^{p,\infty}$ on $\bbR^{\bbT_{n'}}$.
\end{proof}

\begin{proof}[Proof of Corollary~\ref{c:1.17}]
Monotonicity of $\rmE^{+,\delta} h(x_k)$ in $k$ is given by Theorem~\ref{p:107.5}.
Display~\eqref{e:201.50} follows from~\eqref{e:107.39} of the exact same theorem with $A=\{x_k\}$ in combination with the zero mean of $h(x_k)$ under $\rmP_\infty$. Thanks to the tree symmetries, we may assume that $[x_k]_l = x_{k \wedge l}$ for all $k,l \geq 0$. Then the second part of Corollary~\ref{c:1.9}, with $A_n = \{x_k :\: k \leq n'\}$ implies~\eqref{e:1.19a}, uniformly as desired.
\end{proof}

\begin{proof}[Proof of Corollary~\ref{c:1.18}]
Fix $p \geq 1$. Display~\eqref{e:1001.20a} follows from~\eqref{e:107.39} of Theorem~\ref{p:107.5} with $A=\{x_k\}$ and the Triangle Inequality, since $\eta^{+,\delta}(\infty)$ has all moments bounded uniformly in $\delta$.
At the same time, as in the proof of Corollary~\ref{c:1.17}, taking an infinite branch ${\bf x} = (x_k)_{k \geq 1}$ of $\bbT$ starting at the root, and setting $A_n = {\bf x} \cap \bbT_{n'}$, the second part of Corollary~\ref{c:1.9} then implies~\eqref{e:1001.19a}, uniformly as desired.
\end{proof}

\section{Global statistics}
\label{s:4}
\subsection{WLLN framework}
Deriving the asymptotic law of all global features treated here reduce, either directly or indirectly, to studying the distribution (under the conditioning) of an empirical average of a certain statistic of each of the restrictions of the field on the $n$ subtrees rooted at generation $l_n$. The picture which is drawn in the introduction, and partly validated in the previous section, shows that the laws of these restrictions are almost independent of each other. As they are certainly identically distributed, this empirical average should converge to a constant via a weak law of large numbers type argument. The goal of this subsection is therefore to provide a framework for this argument, in the form of a general theorem, which could be used for all cases at hand.

We begin with a (partial) generalization of Lemma~\ref{l:101.7} to sequences of functions.
\begin{lemma}
\label{l:5.5}
Let $(F_n)_{n \in [1,\infty]}$ be a sequence functions from $\bbR$ to $\bbR$ such that
\begin{equation}
\label{e:5.83}
F_n \underset{n \to \infty}\longrightarrow F_\infty \,,
\end{equation}
uniformly on compacts, and that for some $C < \infty$,
\begin{equation}
\label{e:5.84}
|F_n(u)| \leq C\rme^{C|u|} 
\ \ ; \quad n \geq 1,\, u \in \bbR\,.
\end{equation}
Then, with $x_{l_n} \in \bbL_{l_n}$, 
\begin{equation}
\label{e:5.85}
\Big| \rmE_n^+ F_n \big(\hat{h}(x_{l_n})\big) - \rmE^{+,[n]_2} F_\infty(h(0)) \Big| \underset{n \to \infty} \longrightarrow 0\,,
\end{equation}
with the last mean bounded uniformly in $n$.
Moreover, if $\big((F^{(\iota)}_n)_{n \in [1,\infty]} :\: \iota \in \cI \big)$ is a collection of sequences as above, such that~\eqref{e:5.83} holds also uniformly in $\iota$ and ~\eqref{e:5.84} holds for all $\iota$ with the same $C$, then the limit in~\eqref{e:5.85} as well as the bound on the integral thereof also hold uniformly in $\iota$.
\end{lemma}
\begin{proof}
Thanks to the almost-Gaussian tails of $\hat{h}(x_{l_n})$ under $\rmP_n^+$ and those of $h(0)$ under $\rmP^{+,\delta}$, which are a consequence of Theorem~\ref{t:1.3} and Lemma~\ref{l:101.7}, we may replace $F_n(u)$ for $n \in [1,\infty$] by $\tilde{F}_{n,M}(u) := F_n(u)\1_{[-M, M]}(u),$ with both means in~\eqref{e:5.85} changing by an arbitrarily small constant, provided we choose $M$ large enough. Then the uniform convergence of $\tilde{F}_{n,M}$ to $\tilde{F}_{\infty,M}$ together with the second limit in~\eqref{e:107.18} of Lemma~\ref{l:101.7} with $r=0$, shows that~\eqref{e:5.85} holds for the new functions. Uniformity in the case of a collection of sequences, follows from the uniformity of the above estimates, together with the uniformity in $F$ of~\eqref{e:107.18}.
\end{proof}

The next proposition is a weak law of large numbers for functionals of the subtrees of vertices at generation $l_n$ under $\rmP_n^+$. It is used as a key tool in the proofs of the results in this subsection. Recall that $\bbT_k(x)$ for $x \in \bbT$ denotes the subtree of $\bbT$ which includes all descendants of $x$ (including $x$ itself) which are at genealogical distance to $x$ of at most $k$.

\begin{proposition}
\label{p:4.5b} 
For each $n \geq 1$, $x \in \bbL_{l_n}$ let $\xi_n(x)$ be a random variable which is measurable w.r.t. $\sigma(\{h(y) :\: y \in \bbT_{n'}(x)\})$ and has conditional law $\rmP_n^+(\xi_n(x) \in \cdot\, | \,h(x) = u)$ which does not depend on $x$ for all $u \in \bbR$. Suppose also that $(\xi_n(x_n))_{n \geq 1, x_n \in \bbL_{l_n}}$ are uniformly integrable under $\rmP_n^+$.  Then,
\begin{equation}
\label{e:5.5}
\bigg| \frac{1}{|\bbL_{l_n}|} \sum_{x \in \bbL_{l_n}} \xi_n(x)
	- \rmE_n^+ \xi_n(x) \bigg|
\,	{\underset{n \to \infty}\longrightarrow} \, 0 \,,
\end{equation}
in $\rmP_n^+$ probability and in $\bbL^1$.
\end{proposition}
\begin{proof}
Thanks to uniform integrability, the mean under $\rmP_n^+$ of $|\xi_n(x_n)|\1_{\{|\xi_n(x_n)| > u_0\}}$ tends to $0$ as $u_0 \to \infty$ uniformly in $n$ and $x_n \in \bbL_{l_n}$. The same applies to the sum in~\eqref{e:5.5} with the terms therein replaced by $|\xi_n(x_n))|\1_{\{|\xi_n(x_n)| > u_0\}}$. An application of Markov's inequality then shows that we may assume, w.l.o.g., that $\xi_n(x_n)$-s are uniformly bounded, so that $\rmP_n^+$ almost-surely, $|\xi_n(x_n)| \leq M$ for all $n$, $x_n \in \bbL_{l_n}$ and some fixed $M$.

Now, by symmetry the variance of the sum in~\eqref{e:5.5} under $\rmP_n^+$ is equal to 
\begin{equation}
\label{e:5.6}
|\bbL_{l_n}|^{-2} \sum_{x,y \in \bbL_{l_n}} \Cov_n^+ \big(\xi_n(x), \xi_n(y)\big)
= |\bbL_{l_n}|^{-1} \sum_{k=0}^{l_n} 2^k \Cov_n^+ \big(\xi_n(x_n), \xi_n(x_{n,k})\big) \,,
\end{equation}
where $x_n$ is as before and $x_{n,k}$ is such that $|x_{n,k} \wedge x_n| = l_n-k$. 
Since $\hat{h}(x_n)$ and $\hat{h}(x_{n,k})$ are independent conditionally on 
$\hat{h}(x_{n,k} \wedge x_n)$ under $\rmP_n^+$, so are $\xi_n(x_n)$ and $\xi_n(x_{n,k})$. Conditioning then on $\hat{h}(x_{n,k} \wedge x_n)$ and using the total covariance formula and symmetry, the second covariance in~\eqref{e:5.6} is equal to
\begin{equation}
\label{e:5.7}
	\rmVar_n^+ \Big(\rmE_n^+ \Big(\xi_n (x_n)\,\Big|\, \hat{h}(x_{n,k} \wedge x_n)\Big)\Big)\,.
\end{equation}

As in~\eqref{e:5.75}, we may use Proposition~\ref{p:106.3} and the assumed bound on $\xi_n(x)$, to get
\begin{equation}
\Big| \rmE_n^+ \Big(\xi_n(x_n)\,\Big|\, \hat{h}(x_{n,k} \wedge x_n)\Big)\, - \,
	\rmE_n^+ \Big(\xi_n(x_n)\,\Big|\, \hat{h}(x_{n,k} \wedge x_n) = 0\Big)\Big| \leq 
	CM \rme^{-(ck - \hat{h}(x_{n,k} \wedge x_n))^+} \,,
\end{equation}
which bounds the variance in~\eqref{e:5.7} from above by 
\begin{equation}
C^2M^2\rmE_n^+  \rme^{-2(ck - \hat{h}(x_{n,k} \wedge x_n))^+} \leq C' \rme^{-2ck} \,,
\end{equation}
where the last inequality follows from the almost Gaussian tails of $\hat{h}(x_{n,k} \wedge x_n)$, as given by Theorem~\ref{t:1.3}. 

Altogether~\eqref{e:5.6} is at most $C\rme^{-cl_n}$, which in light of Chebyshev's inequality implies~\eqref{e:5.5}.
\end{proof}
Combining Proposition~\ref{p:4.5b} and Lemma~\ref{l:5.5} we get,
\begin{corollary}
\label{c:5.6}
For $n \geq 1$, $x \in \bbL_{l_n}$, let $\xi_n(x)$ be as in Proposition~\ref{p:4.5b} and denote
\begin{equation}
	F_n(u) := \rmE_n^+ \big(\xi_n(x) \,\big|\, h(x) - m_{n'} = u \big) \,.
\end{equation}
Suppose that the conditions of Proposition~\ref{p:4.5b} hold as well as those of 
Lemma~\ref{l:5.5} with some $C > 0$ and some 
$F_\infty: \bbR \to \bbR$. Then,
\begin{equation}
\label{e:5.5a}
\bigg| \frac{1}{|\bbL_{l_n}|} \sum_{x \in \bbL_{l_n}} \xi_n(x)
- \theta_{[n]_2} \bigg|
\,	{\underset{n \to \infty}\longrightarrow} \, 0 \,,
\end{equation}
in $\rmP_n^+$ probability and in $\bbL^1$, where for log-dyadic $\delta \in [0,1)$,
\begin{equation}
\label{e:5.92}
	\theta_\delta := \rmE^{+,\delta} F_\infty \big(h(0)\big) \,,
\end{equation}
which is bounded uniformly in $\delta$. Moreover, if $\xi_n(x)$ is measurable w.r.t. $\sigma(h(x))$, then uniform integrability of $(\xi_n(x))_{n \geq 1}$ under $\rmP_n^+$ holds automatically, once the other conditions in the corollary are satisfied.
Lastly, if $\big(F^{(\iota)}_n)_{n \in [1,\infty]} :\: \iota \in \cI \big)$ is a collection of functions satisfying the same conditions as in Lemma~\ref{l:5.5} and $\xi^{(\iota)}_n(x) := F^{(\iota)}_n(h(x)-m_{n'})$, then~\eqref{e:5.5a} and~\eqref{e:5.92} with $\xi_n(x)$, $F_\infty$, $\theta_\delta$ replaced by $\xi_n^{(\iota)}(x)$, $F_\infty^{(\iota)}$, $\theta_\delta^{(\iota)}$ hold uniformly in $\iota$.
\end{corollary}
\begin{proof}[Proof of Corollary~\ref{c:5.6}]
By the Tower Property,
\begin{equation}
	\rmE_n^+ \xi_n(x) = \rmE_n^+ F_n(h(x) - m_{n'}) \,,
\end{equation}
where $F_n$ is as in the statement of the corollary. The desired statement then follows immediately from Proposition~\ref{p:4.5b}	 and Lemma~\ref{l:5.5}. If $\xi_n(x)$ is measurable w.r.t. $\sigma(h(x))$, then $\xi_n(x) = F_n(h(x) - m_{n'})$. Consequently, uniform integrability in $n$ follows from the exponential bound on $F_n$ together with the almost Gaussian tails of $h(x)-m_{n'}$ under $\rmP_n^+$. This uniform integrability is also w.r.t. $\iota$ once the conditions in Lemma~\ref{l:5.5} hold. Then the limit in Proposition~\ref{p:4.5b} holds uniformly in $\iota$ which combined with Lemma~\ref{l:5.5} show that~\eqref{e:5.5a} and~\eqref{e:5.92} hold uniformly in $\iota$ as well.
\end{proof}

\subsection{Estimates under typical conditioning on the minimum}
Having the WLLN framework setup, we next provide the needed estimates on the statistics which will be averaged. These come in the form of statements about the BRW conditioned on the right tail of its minimum at a typical value, namely under the conditioning on $\Omega_n(u)$ for a fixed $u \in \bbR$. We shall need both bounds and asymptotics. The two main tools in the proofs here are FKG type results, of the form appearing in Subsection~\ref{ss:2.1}, and the following key lemma from~\cite{Work1}.
\begin{lemma}[Lemma~\ref{1@l:103.10} in~\cite{Work1}]
	\label{l:103.10}
	Let $n \geq 1$, $x_n \in \bbL_n$ and $\xi_n$ measurable w.r.t. $h$ on $\bbT_n$. For each $k \in [0,n]$ set $x_k := [x]_k$ and
	\begin{equation}
		\xi_k := \rmE_n \big(\xi_n\,\big|\, h\big(\{x_k\} \cup \{\bbT_n \setminus \bbT(x_k)\} \big) \,.
	\end{equation}
	Then, for all $m \geq 0$, $u \geq 0$,
	\begin{equation}
		\label{e:203.52}
		\Big|\rmE_n^{\uparrow u} \xi_n - \rmE_n^{\uparrow u} \xi_m \Big| 
		\leq C\big(\rme^{C u^2} + 1\big) \sum_{k=m}^n 
		k \rme^{-c_0 k} \rmE_n \Big((h(x_k)^++1) \rme^{-c_0 h(x_k)} |\xi_k|\Big)\,.
	\end{equation}
\end{lemma}

We start with the following simple computation.
\begin{proposition}
	\label{l:2.9}
There exist $C,c \in (0,\infty)$ such that for all $0 \leq k \leq n$, $x \in \bbT_k$, 
$u \in \bbR$ and $v \geq 0$,
\begin{equation}
\label{e:2.8a}
c \frac{1}{v/\sqrt{k} + 1}
\rme^{-\frac{v^2}{2k}}
\leq \rmP^{\uparrow u}_n \big(h(x) > v \big) \leq 
C\rme^{-\frac{(v-u^+)^2}{2k}} \,.
\end{equation}	
Moreover, for all $v \in \bbR$, 
\begin{equation}
\label{e:2.8b}
c  \frac{1}{(v+u^+)/\sqrt{k}+1}
\rme^{-\frac{(v+u^+)^2}{2k}}
\leq \rmP^{\uparrow u}_n \big(h(x) < -v\big) \leq 
C\rme^{-\frac{v^2}{2k}} \,.
\end{equation}	
\end{proposition}

\begin{proof}
The lower bound for the upper tail and the upper bound for lower tail follow by FKG of the law $\rmP^{\uparrow u}_n$. Indeed, conditioning on $\Omega_n(u)$ increases the probability of the upper tail and decreases the probability of the lower tail. Then the bounds follow from standard Gaussian tail asymptotics:
\begin{equation}
	\rmP \big(Z > u) = \frac{1+o(1)}{u\sqrt{2\pi}} \rme^{-u^2/2} \,,
\end{equation}
as $u \to \infty$, with $Z \sim \cN(0,1)$, as $h(x) \sim \cN(0,k)$.

 For the opposite bounds, by Lemma~\ref{lemma:FKG} and Lemma~\ref{l:2.4}, 
\begin{equation}
\label{e:2.40a}
	\rmP_n \big(h(x) > v\,\big| \Omega_n(u) \big) \leq 
	\rmP_n \big(h(x) > v-u^+\,\big| \Omega_n(-u^-) \big) 
\end{equation}
as well as
\begin{equation}
\label{e:2.40b}
\rmP_n \big(h(x) < -v\,\big| \Omega_n(u) \big) \geq 
\rmP_n \big(h(x) < -v-u^+\,\big| \Omega_n(-u^-) \big) \,.
\end{equation}
Since $\rmP_n(\Omega_n(-u^-)) > c$ for all $n \geq 0$ and $u$, thanks to the tightness of the centered minimum, we may consider the intersection of the two events instead of the conditioning in the right hand sides of both~\eqref{e:2.40a} and~\eqref{e:2.40b}. Then the upper bound in~\eqref{e:2.8a} follows immediately. 
\end{proof}

Next, we have the following bound on the mean under the conditioning. This was derived in~\cite{Work1} using the tools described above.
\begin{proposition}[Proposition~\ref{1@l:2.8a} in~\cite{Work1}]
	\label{l:2.8a}
	There exists $C \in (0,\infty)$ such that for all $u \in \bbR$, $n \geq 1$, $x \in \bbT_n$,
	\begin{equation}
		\label{e:2.37a}
		0 \leq \rmE^{\uparrow u}_n \big(h(x)\big) < u^+ + C\,.
	\end{equation}
\end{proposition}

Next, we treat the exponential of the field.
\begin{proposition}
\label{l:2.8b}
Let $\alpha \in [0,c_0]$. There exists $C <\infty$ such that for all $u \in \bbR$, $n \geq 1$ and $x \in \bbL_n$,
\begin{equation}
\label{e:2.45}
	1 \leq \rmE_n^{\uparrow u} \rme^{\alpha h(x) - \frac{n}{2} \alpha^2}
	\leq C\rme^{\alpha u^+} \,.
\end{equation}
Moreover, there exists a positive continuous function $A_\alpha: \bbR \to [1,\infty)$, such that
\begin{equation}
\label{e:2.46}
\rmE_n^{\uparrow u} \rme^{\alpha h(x_n) - \frac{n}{2} \alpha^2} \underset{n \to \infty}\longrightarrow A_\alpha (u), 
\end{equation}	
uniformly in $u$ on compact subsets, where $x_n \in \bbL_n$. Lastly, for all $u \in \bbR$, 
\begin{equation}
\label{e:203.65}
A_\alpha(u) = \lim_{|x| \to \infty} \rmE^{\uparrow u} \rme^{\alpha h(x)-\frac{\alpha^2}{2}|x|} \,.
\end{equation}
\end{proposition}
\begin{proof}
Fix $\alpha \in [0,c_0]$. For~\eqref{e:2.45}, the lower bound follows since the conditional mean is larger than the unconditional one due to FKG. For the opposite inequality, when $u \leq 0$, the left hand side can be bounded by the unconditional mean times a constant, as $\Omega_n(0)$ is uniformly bounded from below, thanks to the tightness of the minimum. When $u \geq 0$, we can use Lemma~\ref{l:2.4} to add $u$ to $h(x)$ in the exponent and condition on $\Omega_n(0)$ instead, which only makes the resulting quantity larger. This gives the upper bound in~\eqref{e:2.45}.

For asymptotics, we proceed similarly to the previous proofs. For each $n \geq 0,$ we let $x_n \in \bbL_n$ and set $x_k := [x_n]_k$ for $0 \leq k \leq n$ and $\xi_n : =\rme^{\alpha h(x_n) - \frac{n}{2} \alpha^2}$. Then $\xi_k=\rme^{\alpha h(x_k) - \frac{k}{2} \alpha^2}$, and the mean on the right hand side of~\eqref{e:203.52} in Lemma~\ref{l:103.10}
is at most $C(k+1) \rme^{-k (c_0-\alpha)^2/2}$. Since $c_0 > (c_0-\alpha)^2/2$ for all $\alpha$ as above, the lemma implies
\begin{equation}
\label{e:103.60}
\lim_{m \to \infty}\limsup_{n \to \infty}
 \big|\rmE^{\uparrow u}_n \rme^{\alpha h(x_n) - \frac{n}{2} \alpha^2} 
 - \rmE_n^{\uparrow u} \rme^{\alpha h(x_m) - \frac{m}{2} \alpha^2} \big| = 0 \,,
\end{equation}
uniformly in $u$ on compacts. By Lemma~\ref{l:5.7} we can replace the second mean by its $n \to \infty$ limit. Thanks also to Proposition~\ref{p:1.45}, we thus get
\begin{equation}
\lim_{m \to \infty} \limsup_{n \to \infty} \bigg|
\rmE_n^{\uparrow u} \rme^{\alpha h(x_n) - \frac{n}{2} \alpha^2}
\,-\,
\rmE^{\uparrow u} \rme^{\alpha h(x_m)-\frac{m}{2}\alpha^2}\bigg| = 0 \,.
\end{equation}
This shows that, uniformly in $u$ on compacts, both sequences of expectations above tend to a finite limit $A_\alpha(u)$ when $n$ and $m$, respectively, tend to $\infty$. Continuity in $u$ in the limit then follows from continuity in same variable of the pre-limits, which is a consequence of, e.g., Proposition~\ref{p:1.45}, Lemma~\ref{l:4.2a} and the Dominated Convergence Theorem.
\end{proof}

For the right tail under the conditioning, we have,
\begin{proposition}
\label{p:2.7a}
There exist $C,c \in (0, \infty)$ such that for all $n \geq 1$,  $u,v \in \bbR$, 
\begin{equation}
\label{e:2.61}
c v\rme^{-c_0 v} \leq 
	\rmP^{\uparrow u}_n \big(\max_{x \in \bbL_n} h(x) - m_n > v \big)
		\leq C\big((v-u^+)^++1\big)  \rme^{-c_0 (v-u^+)^+} \,.
\end{equation}	
Moreover, with $A_\alpha$ as in Proposition~\ref{l:2.8b} and $C_0$ as in Lemma~\ref{l:2.10},
\begin{equation}
\label{e:2.62}
	\rmP_n^{\uparrow u} \big(\max_{x \in \bbL_n} h(x) - m_n > v \big)
		= C_0 A_{c_0}(u) v \rme^{-c_0 v} (1+o(1)) \,,
\end{equation}	
where $o(1) \to 0$ as $n \to \infty$, uniformly in both $v \in (\log \log n,\, (\log \log n)\log n )$ and $u$ in any compact interval. 
\end{proposition}
\begin{proof}
Starting with~\eqref{e:2.61}, the lower bound follows by lower bounding the conditional probability with the unconditional one via FKG and using Lemma~\ref{l:2.10}. For the upper bound in~\eqref{e:2.61}, if $u \leq 0$, the event we condition on has a probability which is bounded away from $0$ uniformly for all $n \geq 1$. The statement follows therefore immediately from the uniform positivity of $\Omega_n(u)$ under $\rmP_n$. If $u > 0$, we can use Lemma~\ref{l:2.4} to replace $v$ by $v-u$ and $u$ by $0$, thereby only increasing the probability. The result then follows from the case $u \leq 0$.

Turning to asymptotics, it is sufficient to show that 
\begin{equation}
\label{e:2.36}
\rmP_n^{\uparrow u} \big(h(x_n) = \max_{\bbL_n} h > m_n + v\big)= A(u) 2^{-n} v \rme^{-c_0 v} (1+o(1)) \,,
\end{equation}	
with $o(1) \to 0$ uniformly in $u,v$ and with $x_n \in \bbL_n$ and $A$ as in the statement of the lemma. Indeed, then summing over all $x_n \in \bbL_n$ gives the desired statement.
To this end, for $0 \leq l,k \leq n$ and $v$ in the range in the statement of the proposition, we let $x_k := [x_n]_k$ and 
\begin{equation}
\xi^{l,v}_n := 2^{n} v^{-1} \rme^{c_0 v} \1_{A_n^{l,v}}
\ ; \quad
A_n^{l,v} := \Big\{h(x_n) = \max \big\{h(z) :\: z \in \bbL_{n-l}(x_{l}) \big\} > m_n + v \Big\}\,.
\end{equation}
We then observe that with $\xi_k^{l,v}$ defined as $\xi_k$ in Lemma~\ref{l:103.10} with $\xi_n^{l,v}$ in place of $\xi_n$, whenever $k \geq l$, 
\begin{equation}
\begin{split}
2^{-n} v & \rme^{-c_0 v} \xi^{l,v}_k \leq \rmP_n \big(h(x_n) = \max_{\bbL_{n-k}(x_{k})} h > m_n + v \,\big|\, h(x_k) \big) \\
& \leq 2^{-(n-k)} \big(q_{n-k}(m_n - m_{n-k} - h(x_k) + v)\big) 
\leq C 2^{-n} v\rme^{-c_0 v} k 2^{-k} (h(x_k)^- +1) \rme^{c_0 h(x_k) } \,.
\end{split}
\end{equation}
Above we used Lemma~\ref{l:2.10}.
Then then mean on the right hand side of~\eqref{e:203.52} (with $\xi_k^{l,v}$ in place of $\xi_k$) is at most
\begin{equation}
C k 2^{-k} \rmE_n (h(x_k)^2+1) \leq C' \,.
\end{equation}
It follows that the sum on the right hand side of~\eqref{e:203.52} with $\xi_k^{m,v}$ in place of $\xi_k$ tends to $0$ as $m \to \infty$, uniformly in $n \geq m$,  $u$ on compacts and $v$ as in the range in the proposition. This gives
\begin{equation}
\label{e:103.74}
 \Big| 	 2^{n} v^{-1} \rme^{c_0 v} \rmP_n^{\uparrow u} \big(A_n^{m,v})
- \rmE_n^{\uparrow u} \Big(2^m v^{-1} \rme^{c_0 v}
	 q_{n-m}\big(m_n - m_{n-m} - h(x_m) + v\big)\Big)\Big| 
	 \longrightarrow 0\,,
\end{equation}
in the above limits. 

At the same time, for fixed $m$, $s$ by Lemma~\ref{l:2.10} and $v$ as above,
\begin{equation}
q_{n-m}(m_n-m_{n-m}-s+v) = C_0 v \rme^{-c_0 (c_0 m + v - s)} (1+o(1)) = 
C_0 2^{-2m} v \rme^{-c_0 (v - s)} (1+o(1)) \,,
\end{equation}
with $o(1) \to 0$ as $n \to \infty$, uniformly in $v$. The same lemma gives for fixed $m$, all $n$ large enough, all $s \in \bbR$ and $v$ as before,
\begin{equation}
2^m v^{-1} \rme^{c_0 v} q_{n-m}(m_n-m_{n-m}-s+v) \leq C(s^- + 1) \rme^{c_0 s^+},
\end{equation}
with $C$ depending on $m$. 
Invoking Lemma~\ref{l:5.7}, Proposition~\ref{l:2.9} and the Dominated Convergence Theorem, the
mean in~\eqref{e:103.74} tends to
\begin{equation}
\label{e:103.77}
	C_0 2^{-m} \rmE^{\uparrow u} \rme^{-c_0 h(x_m)} \,,
\end{equation}
as $n \to \infty$, for fixed $m$ and uniformly in $u$ and $v$ as above.

On the other hand,
\begin{equation}
\label{e:2.66a}
\begin{split}
\rmP_n & \big(A_n^{m,v} \big) - \rmP_n
\Big(h(x_n) = \max_{\bbL_n} h > m_n + v\Big) \\
& \leq \sum_{j=0}^{m-1}
\rmP_n \Big(\max_{\bbL_{n-j}(x_{j}) \setminus \bbL_{n-j-1}(x_{j+1})} h > m_n + v 
\,,\,\,
h(x_n) = \max_{\bbL_{n-j-1}(x_{j+1})} h > m_n + v \Big) \\
& \leq  \sum_{j=0}^m
2^{-(n-j)} \rmE_n \Big(1-\hat{p}_{n-j}(-m_n+m_{n-j}+h(x_j)-v) \Big)^2
\leq C m^2 2^{-n+m} v^2 \rme^{-2c_0 v} 
\,,
\end{split}
\end{equation}	
where we conditioned on $h(x_j)$ to get the second inequality and 
used that the last mean is at most
\begin{equation}
	\rmE_n \big((v+c_0j-h(x_j))^+ + 1)^2 \rme^{-2c_0 (v+c_0 j-h(x_j))^+}
	\leq C\rmE_n \big((v^2+j^2) \rme^{-2c_0 (v+c_0 j-h(x_j))}
	\leq C' v^2(j+1)^2\rme^{-2c_0v} \,.
\end{equation}
In view of the lower bound on $v$, the difference in the probabilities in~\eqref{e:2.66a} is $o(2^{-n} v\rme^{-c_0 v})$, as $n \to \infty$, also for fixed $m$ and uniformly in $v$ as required. The same applies with $\rmP_n$ replaced by $\rmP_n^{\uparrow u}$, now also uniformly in $u$ on compacts, as $\rmP_n(\Omega_n(u))$ is uniformly bounded away from $0$ for such $u$ thanks to the tightness of the centered minimum.

Combining this with~\eqref{e:103.74} and \eqref{e:103.77} and in view of~\eqref{e:2.36} we obtain
\begin{equation}
\lim_{m \to \infty} \limsup_{n \to \infty}
 \Big| v^{-1} \rme^{c_0 v} \rmP_n^{\uparrow u} \Big(\max_{\bbL_n} > m_n + v\Big)
- C_0 2^{-m} \rmE^{\uparrow u} \rme^{-c_0 h(x_m)} \Big|
	 \longrightarrow 0 \,,
\end{equation}
uniformly in $u$ and $v$ as desired.
This readily implies~\eqref{e:2.62} with
\begin{equation}
C_0 \lim_{|x| \to \infty} 2^{-|x|} \rmE^{\uparrow u} \rme^{c_0 h(x)} =
	C_0 \lim_{|x| \to \infty} \rmE^{\uparrow u} \rme^{c_0 h(x)-\frac{c_0^2}{2}|x|} = C_0 A_{c_0}(u)\,,
\end{equation}
where $A_{c_0}$ is as in Proposition~\ref{l:2.8b}.
\end{proof}

\subsection{Proof of main theorems}
We are now ready for
\begin{proof}[Proof of Theorem~\ref{t:1.9}]
We wish to use Proposition~\ref{p:4.5b} with 
\begin{equation}
	\xi_n(x) := \frac{1}{|\bbL_{n'}|} \sum_{y \in \bbL_{n'}(x)} h(y) - m_{n'} \,,
\end{equation}
for $n \geq 1$, $x \in \bbL_{l_n}$.
Clearly, the conditional law of $\xi_n(x)$ given $h(x) = u$ does not depend on $x$ for all $u \in \bbR$. To show uniform integrability of $(\xi_n(x))_{x \in \bbL_{l_n}, n \geq 1},$ we observe that conditional on $\hat{h}(x) = u$ the law of $\xi_n(x)-u$ is the same as that of $\frac{1}{|\bbL_{n'}|} \sum_{y \in \bbL_{n'}} h(y)$ under $\rmP_{n'}(\cdot\,|\, \Omega_{n'}(-u))$. Then, for $u \leq 0$, by Lemma~\ref{l:2.4},
\begin{equation}
\begin{split}
	\rmE^+_n \big((\xi_n(x)^+)^2 \,\big|\, \hat{h}(x) = u \big) & \leq
	2(u^+)^2 + 2 \rmE_{n'} \bigg(\Big(\frac{1}{|\bbL_{n'}|} \sum_{y \in \bbL_{n'}} h(y)\Big)^2 \,\bigg|\, \Omega_{n'}(-u) \bigg) \\
	& \leq 4u^2 + 4 \rmE_{n'} \bigg(\Big(\frac{1}{|\bbL_{n'}|} \sum_{y \in \bbL_{n'}} h(y)\Big)^2 \,\bigg|\, \Omega_{n'}(0) \bigg) \leq 
	4u^2 + C \,,
\end{split}
\end{equation}
where the last inequality follows from Lemma~\ref{l:2.4a} and tightness of the centered minimum of $h$.
On the other hand, by FKG of $\rmP_{n'}$,
\begin{equation}
\begin{split}
	\rmE^+_n \big((\xi_n(x)^-)^2 \,\big|\, \hat{h}(x) = u \big) & \leq
	2(u^-)^2 + 2 \rmE_{n'} \bigg(\Big(\Big(\frac{1}{|\bbL_{n'}|} \sum_{y \in \bbL_{n'}} h(y)\Big)^-\Big)^2 \,\bigg|\, \Omega_{n'}(-u) \bigg) \\
	& \leq 2u^2 + 2 \rmE_{n'} \bigg(\Big(\frac{1}{|\bbL_{n'}|} \sum_{y \in \bbL_{n'}} h(y)\Big)^2 \bigg) \leq 2u^2 + C \,.
\end{split}
\end{equation}
Finally, whenever $u \geq 0$, we use the tightness of the centered minimum (see e.g. \cite{AidekonBRW13}) to write
\begin{equation}
\begin{split}
	\rmE^+_n \big(\xi_n(x)^2 \,\big|\, \hat{h}(x) = u \big) & \leq
	C u^2 + C \rmE_{n'} \bigg(\Big(\frac{1}{|\bbL_{n'}|} \sum_{y \in \bbL_{n'}} h(y)\Big)^2 \bigg) \leq C'u^2 + C \,.
\end{split}
\end{equation}

In all cases, we get the bound $\rmE^+_n \big(\xi_n(x)^2 \,\big|\, \hat{h}(x) = u \big) \leq C(u^2+1)$ for sufficiently large $C > 0$, so that by the uniform tail bounds for the law of $\hat{h}(x)$ under $\rmP_n^+$ as stated in Theorem~\ref{t:1.3}, we get $\rmE^+_n(\xi_n(x)^2) < C$ for all $n$ and $x \in \bbL_{l_n}$.

We may thus invoke Proposition~\ref{p:4.5b} to get
\begin{equation}
\begin{split}
\bigg|\frac{1}{|\bbL_n|} \sum_{y \in \bbL_n} h(y) -
\rmE_n^+ h(y) \bigg|
& = 
\bigg|\frac{1}{|\bbL_n|} \sum_{y \in \bbL_n} h(y) - m_{n'} - 
\big(\rmE_n^+ h(y) - m_{n'} \big) \bigg| \\
& =
\bigg|\frac{1}{|\bbL_{l_n}|} \sum_{x \in \bbL_{l_n}} \xi_n(x)
- \rmE_n^+ \xi_n(x)) \bigg| \underset{n \to \infty} \longrightarrow 0\,,
\end{split}
\end{equation}
in $\rmP_n^+$ probability and in $\bbL^1$.

The second assertion follows immediately by combining the above with the last statement in Corollary~\ref{c:1.17}.
\end{proof}

Next, we give
\begin{proof}[Proof of Theorem~\ref{t:1.8}]
The proof is similar to that of Theorem~\ref{t:1.9}. For $n \geq 1$, $x \in \bbL_{l_n}$, let 
\begin{equation}
\xi_n(x) :=  \rme^{-\alpha m_{n'} - \frac{n'}{2} \alpha^2} \frac{1}{|\bbL_{n'}|} \sum_{y \in \bbL_{n'}(x)} \rme^{\alpha h(y)} =  \rme^{\alpha \hat{h}(x)} \frac{1}{|\bbL_{n'}|}
\sum_{y \in \bbL_{n'}(x)} \rme^{\alpha(h(y) - \hat{h}(x)) - \frac{n'}{2} \alpha^2}  \,. 
\end{equation}
Then, conditional on $\hat{h}(x) = u \in \bbR$, the law of $\rme^{-\alpha u} \xi_n(x)$ is that of
\begin{equation}
\label{e:5.110}
Z_{n',\alpha} := \frac{1}{|\bbL_{n'}|} \sum_{y \in \bbL_{n'}} \rme^{\alpha h(y)-\frac{n'}{2} \alpha^2}  \,,
\end{equation}
under $\rmP_{n'}(-|\Omega_{n'}(-u))$. Then for all $M > 0$, by Lemma~\ref{l:2.4}
and tightness of the centered minimum of $h$, 
\begin{equation}
\begin{split}
	\rmE_n^+ \Big(\xi_n(x) 1_{\{\xi_n(x) > M\}} \,\Big|\, \hat{h}(x) = u \Big) &
\leq
	\rmE_n^+ \Big(\xi_n(x) 1_{\{\xi_n(x) > M\}} \,\Big|\, \hat{h}(x) = u^+ \Big) \\
& \leq
	\rme^{\alpha u^+} \rmE_{n'} \big(Z_{n'\, \alpha} 1_{\{Z_{n'\, \alpha} > M \rme^{-\alpha u^+}\}} \,\big|\, \Omega_{n'}(-u^+) \big)  \\
&	\leq 
	C \rme^{\alpha u^+} \rmE_{n'} \big(Z_{n'\, \alpha} 1_{\{Z_{n'\, \alpha} > M \rme^{-\alpha u^+}\}} \big) \,.
\end{split}
\end{equation}
Now, since $\alpha <\sqrt{2\log 2}$ and using Proposition~\ref{p:2.6}, we may find $\epsilon = \epsilon(\alpha)> 0$ such that $\sup_{n' \geq 1} \rmE_{n'} |Z_{n'\, \alpha}|^{1+\epsilon} < A < \infty$. Then,
\begin{equation}
	\rmE_{n'} \big(Z_{n'\, \alpha} 1_{\{Z_{n'\, \alpha} > M \rme^{-\alpha u^+}\}} \big) 
	\leq \big(\rme^{\alpha \epsilon u^+} M^{-\epsilon}\big) \, \rmE_{n'} |Z_{n'\, \alpha}|^{1+\epsilon} \leq 
	A \rme^{\alpha \epsilon u^+} M^{-\epsilon} \,.
\end{equation}
All together 
\begin{equation}
	\rmE_n^+ \Big(\xi_n(x) 1_{\{\xi_n(x) > M\}} \,\Big|\, \hat{h}(x) = u \Big)
	\leq C \rme^{2\alpha u^+} M^{-\epsilon} \,.
\end{equation}
Thanks to the super-exponentially decaying tail of the law of $\hat{h}(x)$ under $\rmP_n^+$, as shown by Theorem~\ref{t:1.3}, we thus have $\rmE_n^+ \big(\xi_n(x) 1_{\{\xi_n(x) > M\}}\big) \leq c M^{-\epsilon}$, which tends to $0$ as $M \to \infty$, uniformly in $n$ and $x \in \bbL_{l_n}$.

At the same time, in view of~\eqref{e:5.110},
\begin{equation}
	F_n(u) := \rmE_n^+\big(\xi_n(x)\,\big|\, \hat{h}(x) = u\big) = \rme^{\alpha u}
	\rmE_{n'} \Big( \rme^{\alpha h(y)-\frac{n'}{2} \alpha^2} \,\big|\, \Omega_{n'}(-u) \Big) \,,
\end{equation}
is bounded by $C\rme^{\alpha u^+}$ for all $u$ by Proposition~\ref{l:2.8b}. The same proposition also shows that $F_n(u)$ converges to the continuous function $F_\infty(u) := \rme^{\alpha u} A_\alpha(-u)$ as $n \to \infty,$ uniformly in $u$ on compacts .
We may thus apply Corollary~\ref{c:5.6} to conclude that 
\begin{equation}
\bigg| \frac{1}{|\bbL_n|} \sum_{x \in \bbL_n} \exp \Big(\alpha h(x) -\alpha m_{n'}-\frac{n'}{2} \alpha^2\Big) - A_\alpha^{[n]_2} \bigg| \\
= \bigg| \frac{1}{|\bbL_{n'}|} \xi_n(x) - \rmE_n^+ \xi_n(x) \bigg|	 \underset{n \to \infty}\longrightarrow 0 \,,
\end{equation}
in $\rmP_n^+$ probability, 
where for log-dyadic $\delta \in [0,1)$,
\begin{equation}
\label{e:207.100}
A_\alpha^{\delta} := \rmE^{+,\delta} \rme^{\alpha h(0)} A_\alpha(-h(0))  \,.
\end{equation}

Recalling the definition of $A_\alpha$ from Proposition~\ref{l:2.8b}, the last expectation is further equal to
\begin{equation}
\rmE^{+,\delta}
\Big(
\lim_{|x| \to \infty} \rmE^{\uparrow}_{h(0)} \rme^{\alpha h(x)-\frac{\alpha^2}{2}|x|} 
\Big)
= \lim_{|x| \to \infty} \rmE^{+, \delta}  \rme^{\alpha h(x)-\frac{\alpha^2}{2}|x|} \,,
\end{equation}
where we used~\eqref{e:107.28} from Proposition~\ref{p:1.45} and~\eqref{e:107.31} from Theorem~\eqref{t:1.7}. Exchangeability of the order of limit and integration follows from the Dominated Convergence Theorem, which we may apply thanks to the almost Gaussian tails of $h(0)$ as given by Proposition~\ref{p:1.11}. This, together with the positivity and exponential upper bound on $A_\alpha$, which is implied by the first part of Proposition~\ref{l:2.8b}, shows that $\theta_\delta^{\alpha}$ is bounded away from $0$ and $\infty$, uniformly in $\alpha$ and $\delta.$
\end{proof}	

Next we have, 
\begin{proof}[Proof of Theorem~\ref{t:1.5}]
Let $n \geq 1$, fix $u \in \bbR$ and for all $v \in \bbR$, let
\begin{equation}
	p_n(u,v) := \rmP_n \Big(\max_{x \in \bbL_n} h(x) > m_n + u \,\Big|\, \Omega_n(v)\Big) \,.
\end{equation}
Set also for some arbitrary $c>0$,
\begin{equation}
\label{e:5.100}
u_n := c_0^{-1} ((\log 2) l_n + \log l_n) + u 
	= c_0^{-1} (\log_2 n + \log \log_2 n) + u + o(1) \,,
\end{equation}
\begin{equation}
w_n := (\log |\bbL_{l_n}|)^{1/2+c} = (\log n)^{1/2+c} \,.
 \end{equation}
Finally, for $x \in \bbL_{l_n}$, set
\begin{equation}
\xi_n(x) := F_n\big(\hat{h}(x)\big)
\quad; \qquad 
F_n(w) := -2^{l_n}\log \big(1-p_{n'}\big(-w + u_n,\, -w \big) \big) 
	1_{\{w \leq |w_n|\}} \,.
\end{equation}
Then by conditioning on the values of $h$ in generation $l_n$ we have 
\begin{equation}
\label{e:5.2} 
\bigg|\rmP_n^+ \Big(\max_{x \in \bbL_n} h(x) - 2m_{n'} \leq u_n  \Big) 
	- \rmE_n^+ \bigg[ \exp \Big(-2^{-l_n}\sum_{x \in \bbL_{l_n}} \xi_n(x) 
	 \Big) \bigg] \bigg|
	 \leq \rmP_n^+ \Big(\exists x \in \bbL_{l_n} :\: \big|\hat{h}(x)\big| > w_n \Big)\,.
\end{equation}
Now, thanks to the almost Gaussian tail of $\hat{h}(x)$ for $x \in \bbL_{l_n}$ under $\rmP_n^+$, as stated in Theorem~\ref{t:1.3}, in combination with the Union Bound, the right hand side of~\eqref{e:5.2} tends to $0$ as $n \to \infty$. At the same time, by the first part of Proposition~\ref{p:2.7a}, for $w < w_n$, 
\begin{equation}
\label{e:5.117}
	p_{n'}\big(-w + u_n,\, -w \big)	
		\leq C (u_n - w^++1) \rme^{-c_0 (u_n - w^+)}
		\leq C 2^{-l_n} \rme^{-c_0 (u-w^+)} \,,
\end{equation}
where we also used that
\begin{equation}
	u_n \rme^{-c_0 u_n} = c_0^{-1} (\log 2) 2^{-l_n} \rme^{-c_0 u} (1+o(1)),
\end{equation}
as $n \to \infty$, uniformly in $u$ on compacts. Since the right hand side~\eqref{e:5.117} is vanishing as $n\to \infty$ and using that $\log (1-s) = -s + o(1)$ as $s \to 0$, we get 
\begin{equation}
\label{e:5.101}
0 \leq F_n(w) \leq C \rme^{-c_0(u-w^+)} \,.
\end{equation}
The second part of Proposition~\ref{p:2.7a} and the same approximation of the log also shows that
\begin{equation}
F_n(w) \underset{n \to \infty}\longrightarrow F_\infty(w) := C_0c_0^{-1}(\log 2) A_{c_0}(-w) \rme^{-c_0 (u-w)}\,,
\end{equation}
uniformly in $w$ on compacts, with the limit being continuous in $w$ thanks to the continuity of $A_{c_0}$.

We may therefore invoke Corollary~\ref{c:5.6} and, thanks to the positivity of $F_n$, use the Bounded Convergence Theorem to equate the mean in~\eqref{e:5.2} with
\begin{equation}
	\exp \big(-\theta_{[n]_2} \rme^{-c_0 u}\big) + o(1) \,,
\end{equation}
where $o(1) \to 0$ as $n \to \infty$ and, 
\begin{equation}
	\theta_\delta := C_0 c_0^{-1} (\log 2)\, \rmE^{+,\delta} \big(A_{c_0}(-h(0)) \rme^{c_0 h(0)}\big) = C_0 c_0^{-1} (\log 2) A_{c_0}^{\delta} \,,
\end{equation}
with $A_{c_0}^{\delta}$ as in~\eqref{e:207.100}.

Collecting all terms, we obtain
\begin{equation}
\rmP_n^+ \big(\max_{x \in \bbL_n} h(x) - 2m_{n'} \leq u_n \big) 	= 
	\exp \big(-\theta_{[n]_2} \rme^{-c_0 u}\big) + o(1) \,,
\end{equation}
with $o(1) \to 0$ as $n \to \infty$. Plugging in $u = c_0^{-1} \log \theta_{[n]_2}+ u',$ for any $u' \in \bbR$, we recover the desired statement, with $c_1 := \log c_0 - \log C_0 + \log \log 2$.
\end{proof}

\begin{proof}[Proof of Theorem~\ref{thm:minimum}]
Let $n \geq 1$, $u \geq 0$ and set $u_n := 2^{-l_n} u$. For $s \in \bbR, t \geq 0$, set
\begin{equation}
	q_n(t,s):= \rmP_n\Big(\Omega_n(t+s)\,\big|\, \Omega(s)\Big)=\frac{p_n(t+s)}{p_n(s)} \,.
\end{equation}
Then,
\begin{equation}
	\rmP_n^+\Big(\min_{x\in \bbL_n}h(x)\leq u_n \Big)=1-\rmP_n^+\Big(\min_{x\in \bbL_n}h(x)> u_n \Big)
	=1-\rmE_n^+\bigg(\prod_{x\in \bbL_{l_n}} q_{n^\prime}(u_n,-\hat{h}(x)) \bigg) \,.
\end{equation}
Setting for $x\in \bbL_{l_n}$,
\begin{equation}
\xi_n(x):= F_n(\hat{h}(x))	
\quad, \qquad
F^{(u)}_n(v) := -2^{l_n} \log \frac{p_{n^\prime}(u_n-v)}{p_{n^\prime}(-v)} \,,
\end{equation}
the last mean can be rewritten as
\begin{equation}
\label{e:5.122}
	\rmE_n^+ \exp\Big( -2^{-l_n} \sum_{x\in \bbL_{l_n}} \xi_n(x) \Big) \,.
\end{equation}
Now, by the Mean Value Theorem and Lemma~\ref{l:4.2a},
\begin{equation}
F^{(u)}_n(v) \leq 2^{l_n} u_n \Big|\sup_{\theta \in [0,1]} \frac{\rmd}{\rmd s} \log p_{n^\prime}(s)\big|_{s=-v+ \theta u_n} \Big| \leq Cu(|v|+1) \,.
\end{equation}
Also by the Mean Value Theorem and, this time, Lemma~\ref{l:2.9a},
\begin{equation}
	F^{(u)}_n(v) \underset{n \to \infty} \longrightarrow
		F^{(u)}_\infty(v) := -u \frac{\rmd}{\rmd s} \log p_{\infty}(s)\big|_{s=-v} \,,
\end{equation}
uniformly in $v$ on compacts, with the limit being continuous in $v$.

Invoking Corollary~\ref{c:5.6} and, thanks to the positivity of $\xi_n(x)$, using the Bounded Convergence Theorem, the mean in~\eqref{e:5.122} is equal to $\rme^{-u\theta'_{[n]_2}} + o(1)$, where $o(1) \to 0$ as $n \to \infty$ uniformly in $u$ on compacts, and
\begin{equation}
\theta'_\delta := -\rmE^{+,\delta} \bigg(\frac{\rmd}{\rmd s} \log p_{\infty}(s)\big|_{s=-h(0)}\bigg) \,.
\end{equation}
The last quantity is bounded away from $0$ and $\infty,$ uniformly in $\delta\in [0,1)$, thanks 
to Lemma~\ref{l:5.5}, Lemma~\ref{l:4.2a} and Lemma~\ref{l:2.9a}.
We thus obtain
\begin{equation}\label{eq:5.85}
	\rmP_n^+\Big(2^{l_n} \min_{x\in \bbL_n}h(x)\leq u \Big)
		= 1-\rme^{-u\theta'_{[n]_2}} + o(1) \,,
\end{equation}
with $o(1) \to 0$, as $n\to \infty$, uniformly in $u$ on compacts.
Plugging in $u = (\theta'_{[n]_2})^{-1} u'$ for any $u' \geq 0$ and setting $\kappa_\delta := \theta'_\delta 2^{-\delta},$ we recover~\eqref{e:1.28}. Since for $u < 0,$ the probability in~\eqref{e:1.28} is trivially zero, the second statement follows immediately.
\end{proof}

\section*{Acknowledgments}
This research was supported through the programme "Research in Pairs"  by the Mathematisches Forschungsinstitut Oberwolfach in 2020 and also partly funded by the Deutsche Forschungsgemeinschaft (DFG, German Research Foundation) under Germany's Excellence Strategy - GZ 2047/1, Projekt-ID 390685813 and GZ 2151 - Project-ID 390873048,
through the Collaborative Research Center 1060 \emph{The Mathematics 
of Emergent Effects}. 
The research of M.F. was also partly supported by a Minerva Fellowship of the Minerva Stiftung Gesellschaft fuer die Forschung mbH. The research of L.H. was supported by the Deutsche Forschungsgemeinschaft (DFG, German Research Foundation)  through Project-ID 233630050 - TRR 146, Project-ID 443891315 within SPP 2265, and Project-ID 446173099.
The research of O.L. was supported by the ISF grant no.~2870/21, and by the BSF award 2018330.
The authors would like to thank Amir Dembo for useful discussions.

\bibliographystyle{abbrv}
\bibliography{bbm-AL-ref.bib}
\end{document}